\documentclass{amsart}
\usepackage[left=1.2in,right=1.2in,tmargin = 25mm,bmargin = 25mm]{geometry}
\usepackage[dvipsnames]{xcolor}
\usepackage[hypertexnames = false,
            colorlinks = true,
            linkcolor = MidnightBlue,
            urlcolor  = NavyBlue,
            citecolor = RedViolet,
            anchorcolor = JungleGreen]{hyperref}
\usepackage{amsthm}
\usepackage{thmtools}
\usepackage{cleveref}
\usepackage{tgpagella} 
\usepackage[utf8x]{inputenc}
\usepackage[T1]{fontenc}
\usepackage{graphicx} 
\usepackage{bm}
\usepackage{wrapfig} 
\usepackage{pgf} 
\usepackage{tikz-cd}
\usetikzlibrary{patterns}
\usepackage{bbm} 
\usepackage{fancyvrb} 
\usepackage{bm}
\usepackage{stmaryrd} 
\usepackage{comment}
\usepackage{ bbold }
\usepackage{enumitem}
\usepackage[main=english,greek]{babel}

\theoremstyle{definition}
\newtheorem{definition}{Definition}[section]
\newtheorem{remark}[definition]{Remark}

\newtheorem{example}[definition]{Example}

\theoremstyle{plain}
\newtheorem{proposition}[definition]{Proposition}
\newtheorem{lemma}[definition]{Lemma}
\newtheorem{theorem}[definition]{Theorem}
\newtheorem{corollary}[definition]{Corollary}

\newcommand{\CC}{\mathcal{C}}
\newcommand{\DD}{\mathcal{D}}

\newcommand{\HH}{\mathcal{H}}
\newcommand{\II}{\mathcal{I}}

\newcommand{\RR}{\mathcal{R}}

\newcommand{\N}{\mathbb{N}}

\newcommand{\Unit}{\mathsf{Unit}}
\newcommand{\tp}{\mathsf{tp}}
\newcommand{\Tm}{\mathsf{Tm}}
\newcommand{\Ty}{\mathsf{Ty}}
\newcommand{\lam}{\mathsf{lam}}
\newcommand{\unlam}{\mathsf{unlam}}
\newcommand{\fst}{\mathsf{fst}}
\newcommand{\fstProj}{\mathsf{fstProj}}
\newcommand{\snd}{\mathsf{snd}}
\newcommand{\sndProj}{\mathsf{sndProj}}
\newcommand{\refl}{\mathsf{refl}}
\newcommand{\Id}{\mathsf{Id}}

\newcommand{\var}{\mathsf{var}}
\newcommand{\compDom}{\mathsf{compDom}}
\newcommand{\pcomp}{\triangleright}
\newcommand{\pair}{\mathsf{pair}}

\newcommand{\Psh}{\mathsf{Psh}}
\newcommand{\sSet}{\mathsf{sSet}}
\newcommand{\iso}{\cong}
\newcommand{\id}{\mathsf{id}}
\newcommand{\op}{\mathsf{op}}
\newcommand{\Cat}{\mathsf{Cat}}
\newcommand{\Grpd}{\mathsf{Grpd}}
\newcommand{\Top}{\mathsf{Top}}

\renewcommand{\gg}{\mathbin{\text{\normalfont{»}}}}

\newcommand{\unit}{\mathsf{unit}}
\newcommand{\counit}{\mathsf{counit}}

\newcommand{\done}[1]{\href{#1}{$\boxbar$}}
\newcommand{\mathlib}[1]{\href{#1}{$\boxbar$}}

\title{Polynomial functors in $\pi$-clans
  for the semantics of type theory}
\author{Joseph Hua, Yiming Xu}
\date{}

\begin{document}
\begin{abstract}
  The category of contexts underlying a model of
  Martin-L\"of type theory with $\Unit$-, $\Sigma$-, and $\Pi$-types need not be locally Cartesian closed,
  but is necessarily a $\pi$-clan.
  We exploit this $\pi$-clan structure to build the theory of polynomial functors.
  This paper presents two equivalent notions of strict semantics for MLTT in this weaker setting,
  respectively \emph{elementary} models --
  reformulating categories with families (CwFs) --
  and \emph{algebraic} models -- reformulating natural models.
  These components fit into a practical sequence of steps for constructing models of MLTT: building an elementary model,
  extracting a $\pi$-clan from the elementary model,
  and then using polynomial functors built on the $\pi$-clan structure
  to convert the elementary model into an algebraic one.
\end{abstract}

\maketitle

\tableofcontents

\section{Introduction}
\subsection{Overview}
In this paper, we develop the theory of polynomial functors in $\pi$-clans \cite{joyal2017},
and investigate their role in the semantics of Martin-L\"of type theory (MLTT). 
We define two notions of strict semantics
(judgmentally equal types are interpreted as equal objects),
respectively called \emph{elementary} and \emph{algebraic} semantics.
The latter makes extensive use of polynomial functors.
We construct translations between the two notions,
showing that a model of MLTT can be presented in either style.

The theory of polynomial functors (or \emph{containers})
forms a general framework for
working with parametrized data in an abstract and organized way.
The diverse applications of this theory 
otherwise unaddressed in this paper
include the semantics of W-types \cite{dybjer1997} \cite{moerdijk2000},
generic programming \cite{abbott2003},
dynamical systems \cite{niu2024},
group cohomology \cite{tambara1993},
Witt vectors and cobordism \cite{brun2005} \cite{brun2007} \cite{bisson2025}, to name but a few.
Polynomial functors play a key role in Awodey's natural models \cite{awodey2018}.

Polynomial functors are most well-studied in
locally Cartesian closed (LCC) categories \cite{gambino2013}:
if $\CC$ is LCC, each map $f:E\to B$ in $\CC$
acts as the signature for a polynomial functor
$P_f : \CC \to \CC$.
Weber \cite{weber2015} generalizes this theory to work
in a category with pullbacks,
where polynomial signatures (\Cref{def:poly}) are required to be exponentiable maps.
This generalization admits examples such as $\Cat$ and $\Top$,
neither of which is LCC.
However,
an initial model of MLTT (with $\Unit$-, $\Sigma$- and $\Pi$-types)
is \emph{not} an example
of Weber's generalization:
context extensions $\Gamma . A \to \Gamma$ are only
exponentiable \emph{relative to other context extensions}.
This leads us to consider \emph{$\pi$-clans}
as an appropriate setting for polynomials.

The notion of a $\pi$-clan can be traced back to Paul Taylor \cite{taylor1987},
who called them \emph{classes of displays},
and later Joyal \cite{joyal2017},
from whom the terminology originates.
Both works justify $\pi$-clans as the \emph{minimal} categorical structure
required for a weak interpretation
(judgmentally equal types are interpreted as isomorphic objects)
of MLTT with $\Unit$-, $\Sigma$-, and $\Pi$-types.
Our work deals with strict models,
which, in particular are weak
and necessarily determine a $\pi$-clan.

In this paper, we define \emph{elementary semantics},
which are similar to CwFs \cite{clairambault2011}
in that the description of type formers closely resembles
a syntactic presentation of MLTT.
Each judgment corresponds to some explicit semantic fact
about arrows in the category of contexts.
As such, writing an interpretation function for MLTT syntax into elementary semantics is as straightforward as one could hope\footnote{
The domain specific language, interpretation function,
and certifying typechecker presented in SynthLean \cite{nawrocki2026}
has now been modified to target elementary semantics
instead of natural models.
In this paper, we will only discuss semantics
and not present this aspect of the overall project.
}.
Elementary semantics do not require a theory of polynomial functors
or $\pi$-clans.

On the other hand, \emph{algebraic semantics} are 
a reformulation of natural models:
admissibility of each type former can be concisely expressed
as the existence of a pullback square\footnote{
  $\Id$-types are not in fact expressed as a pullback square,
  but this is addressed by using $\mathsf{Path}$-types in \cite{awodey2026}.
} corresponding to a cartesian natural transformation between polynomial functors.
To reap the distinct benefits of elementary and algebraic semantics,
we provide translations back and forth between the two.
These results validate the correctness of our definition of algebraic models,
which is the main contribution of this paper.

\subsection{Formalization and design considerations}
This paper is part of the wider HoTTLean project \cite{hua2025},
where a domain-specific language for MLTT is built in Lean
and applied to the groupoid model.
As such, many of our results are formalized in Lean.
The definitions and theorems with formalized counterparts
are labeled with symbols $\boxbar$,
which contain embedded links 
to code in the HoTTLean repository and Mathlib \cite{mathlib2020}.
The HoTTLean repository can be found at
\begin{center}
\href{https://github.com/sinhp/HoTTLean/tree/TYPES2026}{\textsf{github.com/sinhp/HoTTLean/tree/TYPES2026}}
\end{center}
The original goal of this Lean library
was to simplify the formulation of the groupoid model,
as part of the HoTTLean project.
We originally defined the groupoid model using the theory of
natural models.
After experimenting with various setups,
including using Uemura's representable map categories \cite{uemura2023},
we settled on our current approach.
We outline our reasons for this decision in the following.
\begin{itemize}
  \item Let $\CC$ denote a category of contexts.
    HoTTLean is designed to handle multiple small
    (meaning in $\CC$) universes.
    Also adding a large (meaning in $\Psh(\CC)$ but not $\CC$)
    classifier of all types --
    as is the approach in natural models --
    creates an additional burden for the user.
    Our first observation is that one can get away
    with only working with small universes.
    Our second observation is that when doing so,
    one ends up working only with representable presheaves.
    This requires a proof that representables are closed under
    polynomial operations,
    and the theory of polynomials in $\pi$-clans can be seen as
    a systematic way of proving these closure conditions.
    Finally, our work simplifies expressions
    by doing away with
    presheaves and Yoneda embeddings,
    instead working directly in $\CC$.
  \item The move from $\Psh(\CC)$ to $\CC$ 
    brings us closer to an \emph{internal language} for $\CC$.
    Indeed, HoTTLean aims to build a domain-specific language
    for reasoning about the objects of $\CC$,
    rather than $\Psh(\CC)$.
  \item Another advantage of working in $\CC$
    over $\Psh(\CC)$ is in dealing with issues of size.
    Lean's universe level calculus for $\Psh(\CC)$ involves
    two universe variables for $\CC$ (one for the objects and one for the maps) and another for the category
    $\mathsf{Type} \, u$ that the presheaves take values in,
    and these three variables are constrained in certain ways
    that the user needs to determine.
  \item For the groupoid \cite{hofmann1995} and simplicial \cite{kapulkin2021} models of MLTT,
    Uemura's representable map categories provide a sufficiently convenient
    framework.
    Indeed, $\sSet$ is LCC and
    (split) isofibrations are exponentiable in
    $\Grpd$ \cite{vidmar2018}.
    However, to work with the initial model of MLTT
    in Uemura's setting,
    one would have to again take presheaves
    since fibrations are not exponentiable maps.
\end{itemize}
\subsection{Contributions}
Our main contributions are as follows:
\begin{itemize}
    \item We develop the theory of polynomial functors in $\pi$-clans, which is more general than the typical LCC setting.
    \item We apply our theory of polynomial functors in $\pi$-clans
    to the semantics of MLTT, 
    namely, to provide an algebraic presentation of the type formers in the style of Martin-L\"of algebras \cite{awodey2025}.
    \item We formalize the results from this paper that are relevant 
    to the HoTTLean project.
\end{itemize}

\subsection{Structure of the paper}
The paper is structured as follows. 
Starting from \Cref{sec:UnstrSem},
we define \emph{elementary semantics}.
In \Cref{sec:Poly},
we define $\pi$-clans and present polynomial functors in this setting.
Using polynomial functors,
in \Cref{sec:StrSem} we define \emph{algebraic semantics}.
Finally, we define translations back and forth between elementary
and algebraic semantics in \Cref{sec:Translation}.

\subsection{Acknowledgment}
We are indebted to Steve Awodey for suggesting this
line of inquiry in the first place,
and for his continued guidance and insight throughout.
We have also benefitted greatly from discussions with
Mario Carneiro, Sina Hazratpour,
Wojciech Nawrocki, Shuge Rong and Spencer Woolfson from the
HoTTLean project,
with whom we have put these ideas to the test.
We also thank Reid Barton for discussions on
further generalizations of our work.
This material is based on work supported by the Air Force Office of Scientific Research under MURI award number FA9550-21-1-0009.
The second author, Yiming Xu, is co-funded by the European Union (ERC, Nekoka, 101083038).
 Views and opinions expressed are however those of the authors only and do not
necessarily reflect those of the European Union or the European Research Council. Neither
the European Union nor the granting authority can be held responsible for them.

\section{Elementary semantics}\label{sec:UnstrSem}
In the following,
we fix a category $\CC$ for interpreting contexts as objects and
substitutions as morphisms.

\begin{definition}[\done{https://github.com/sinhp/HoTTLean/blob/TYPES2026/HoTTLean/Model/Unstructured/UnstructuredUniverse.lean\#L19}]
  A \emph{universe}
  in $\CC$ consists of a morphism
  $\tp : \Tm \to \Ty$
  and a chosen pullback $\Gamma . A$
  for each object $\Gamma$ and morphism $A : \Gamma \to \Ty$.
  \[\begin{tikzcd}
	{\Gamma.A} & \Tm \\
	\Gamma & \Ty
	\arrow["\var_A", from=1-1, to=1-2]
	\arrow["d_A", swap, from=1-1, to=2-1]
	\arrow["\lrcorner"{anchor=center, pos=0.125}, draw=none, from=1-1, to=2-2]
	\arrow["\tp", from=1-2, to=2-2]
	\arrow["A"', from=2-1, to=2-2]
  \end{tikzcd}\]
  We will view $A : \Gamma \to \Ty$,
  as a type in context $\Gamma$,
  and view the chosen pullback $\Gamma . A$
  as context extension $d_A : \Gamma . A \to \Gamma$.
  We also view a morphism $a : \Gamma \to \Tm$
  satisfying $a \gg \tp = A$ as a term of type $A$.\footnote{
    To match the formalization in Lean,
    we adopt a notation style similar to that used in the Lean library.
    In particular,
    the order of morphism composition will follow the ``arrows'' convention, denoted $f \gg g$ rather than $g \circ f$.
  }
\end{definition}

One can view $\tp : \Tm \to \Ty$ as the context extension
$d : (X : \Ty.\,x : X) \to (X : \Ty)$,
by taking the type $A : \Ty \to \Ty$ to be the identity
$\id_\Ty : \Ty \to \Ty$.

\begin{remark}
  In the following,
  suppose $\CC$ is a category of contexts for the initial model of MLTT,
  namely the syntactic category of contexts and substitutions.
  We compare natural models in $\Psh(\CC)$ with
  universes in $\CC$.
  
  Natural models $\dot{U} \to U$ are allowed to be ``large'',
  where the base $U$ may not be represented by an object in $\CC$.
  Consequently,
  $U$ does not necessarily correspond to any
  small universe from the syntax.
  This also means that a natural model can classify
  \emph{all dependent families} in $\CC$.
  On the other hand,
  a universe $\Tm \to \Ty$
  in the category of contexts must originate from some small universe appearing in the syntax,
  and can only classify ``small'' dependent families.

  A natural model is assumed to be a ``representable'' natural transformation,
  meaning that whenever $\Gamma$ is a representable presheaf,
  $\Gamma . A$ is representable, too.
  This ensures that contexts are closed under the operation of context extension.
  In our setting,
  we place no such condition on a universe,
  since all objects are contexts.
\end{remark}

\begin{definition}[\done{https://github.com/sinhp/HoTTLean/blob/TYPES2026/HoTTLean/Model/Unstructured/UnstructuredUniverse.lean\#L98}]
  For $\sigma : \Delta \to \Gamma$,
  a type $A : \Gamma \to \Ty$ and a term $a : \Delta \to \Tm$
  such that $a \gg \tp = \sigma \gg A$,
  we denote map into the context extension
  using the notation
  $\sigma . a : \Delta \to \Gamma . A$,
  induced by the pullback
\[\begin{tikzcd}
	\Delta \\
	& {\Gamma.A} & \Tm \\
	& \Gamma & \Ty
	\arrow["{\sigma.a}"{description}, dashed, from=1-1, to=2-2]
	\arrow["a", bend left, from=1-1, to=2-3]
	\arrow["\sigma"', bend right, from=1-1, to=3-2]
	\arrow[from=2-2, to=2-3]
	\arrow[from=2-2, to=3-2]
	\arrow["\lrcorner"{anchor=center, pos=0.125}, draw=none, from=2-2, to=3-3]
	\arrow["\tp", from=2-3, to=3-3]
	\arrow["A"', from=3-2, to=3-3]
\end{tikzcd}\]
\end{definition}

\begin{definition}[$\Unit$-types] \label{def:elementary-unit}
  An \emph{elementary $\Unit$-type structure} on a universe
  $\tp : \Tm \to \Ty$
  consists of the following
  \begin{enumerate}
  \item
    For any context $\Gamma$, there is a type $\Unit_\Gamma : \Gamma \to \Ty$.
  \item
    The construction $\Unit_\Gamma$ is stable under substitution,
    meaning that for any $\sigma : \Delta \to \Gamma$ we have
    \[
      \sigma \gg \Unit_\Gamma = \Unit_\Delta
    \]
  \item
    For any context $\Gamma$, there is a unique map $\unit_\Gamma : \Gamma \to \Tm$
    satisfying $\unit_\Gamma \gg \tp = \Unit_\Gamma$.
  \end{enumerate}
\end{definition}

\begin{lemma} \label{lem:elementary-unit-substitution-stable}
  The construction $\unit_\Gamma$ is stable under substitution,
  meaning that for any $\sigma : \Delta \to \Gamma$ we have
  \[ \sigma \gg \unit_\Gamma = \unit_\Delta \]
\end{lemma}
\begin{proof}
  This automatically follows from uniqueness of $\unit_\Delta$
  as a map satisfying $\unit_\Delta \gg \tp = \Unit_\Delta$:
  \[
    \sigma \gg \unit_\Gamma \gg \tp
    = \sigma \gg \Unit_\Gamma
    = \Unit_\Delta
  \]
\end{proof}

\begin{definition}[$\Pi$-types, \done{https://github.com/sinhp/HoTTLean/blob/TYPES2026/HoTTLean/Model/Unstructured/UnstructuredUniverse.lean\#L332}] \label{def:elementary-pi}
  An \emph{elementary $\Pi$-type structure}
  on a universe $\tp : \Tm \to \Ty$
  consists of the following
  \begin{enumerate}
  \item 
    For any context $\Gamma$, type $A : \Gamma \to \Ty$,
    and dependent type $B : \Gamma . A \to \Ty$
    there is a type $\Pi_A B : \Gamma \to \Ty$.
  \item
    The construction $\Pi_A B$ is stable under substitution,
    meaning that for any $\sigma : \Delta \to \Gamma$ we have
    \[
      \Pi_{\sigma \gg A} (\tilde{\sigma} \gg B) =
      \sigma \gg \Pi_{A} B
    \]
    where $\tilde{\sigma} := (d_{\sigma \gg A} \gg \sigma) . \var_{\sigma \gg A}$ as in
    \[
    \begin{tikzcd}
      {\Delta . {(\sigma \gg A)}} & {\Gamma . A} & \Tm \\
      \Delta & \Gamma & \Ty
      \arrow["{{\tilde{\sigma}}}", dashed, from=1-1, to=1-2]
      \arrow["{\var_{\sigma \gg A}}", bend left, from=1-1, to=1-3]
      \arrow["{d_{\sigma \gg A}}", from=1-1, to=2-1]
      \arrow["{\var_A}", from=1-2, to=1-3]
      \arrow["{d_A}", from=1-2, to=2-2]
      \arrow["\lrcorner"{anchor=center, pos=0.125}, draw=none, from=1-2, to=2-3]
      \arrow[from=1-3, to=2-3]
      \arrow["\sigma"', from=2-1, to=2-2]
      \arrow["A"', from=2-2, to=2-3]
    \end{tikzcd}
    \]
  \item
    Under the assumptions of (1),
    for any $b : \Gamma . A \to \Tm$
    such that $b \gg \tp = B$,
    there is a map $\lam \, b : \Gamma \to \Tm$
    satisfying $\lam \;b \gg \tp = \Pi_A B$.
  \item
    The construction $\lam \, b$ is stable under substitution.
    For any $\sigma : \Delta \to \Gamma$
    \[
      \lam (\tilde{\sigma} \gg b) =
      \sigma \gg \lam \, b
    \]
  \item
    Under the assumptions of (1),
    for any $f : \Gamma \to \Tm$
    such that $f \gg \tp = \Pi_A B$,
    there is a map $\unlam \, f : \Gamma . A \to \Tm$
    satisfying $\unlam \, f \gg \tp = B$.
  \item ($\beta$- and $\eta$-rules) The operations 
    $\lam$ and $\unlam$ are inverses of one another.
    \[ \unlam (\lam \, b) = b \quad \text{ and } 
    \quad \lam (\unlam \, f) = f\]
  \end{enumerate}
\end{definition}

\begin{lemma}[\done{https://github.com/sinhp/HoTTLean/blob/TYPES2026/HoTTLean/Model/Unstructured/UnstructuredUniverse.lean\#L360}]
Stability of $\unlam$ can be derived from $\lam$ being an inverse
that is stable under substitution.
\end{lemma}
\begin{proof}
\begin{align*}
  & \unlam (\sigma \gg f) \\
  = \, & \unlam (\sigma \gg \lam (\unlam \, f)) \\
  = \, & \unlam (\lam (\tilde{\sigma} \gg \unlam \, f)) \\
  = \, & \tilde{\sigma} \gg \unlam \, f
\end{align*}
\end{proof}

\begin{lemma}\label{lem:substituted-lam}
  Suppose $\tp$ has elementary $\Pi$-types,
  and fix maps $A : \Gamma \to \Ty$,
  $B : \Gamma . A \to Ty$ and $\sigma : \Delta \to \Gamma$.
  There is a bijection between commutative squares of the following form
  \begin{equation}
    \begin{tikzcd}
    {\Delta . (\sigma \gg A)} & \Tm && \Delta & \Tm \\
    {\Gamma . A} & \Ty && \Gamma & \Ty
    \arrow["b", from=1-1, to=1-2]
    \arrow["{\tilde \sigma}"', from=1-1, to=2-1]
    \arrow["\tp", from=1-2, to=2-2]
    \arrow["f", from=1-4, to=1-5]
    \arrow["\sigma"', from=1-4, to=2-4]
    \arrow["\tp", from=1-5, to=2-5]
    \arrow["B"', from=2-1, to=2-2]
    \arrow["{\Pi_A B}"', from=2-4, to=2-5]
  \end{tikzcd}
  \end{equation}
  given by $b \mapsto \lam \, b$ and $\unlam \, f \mapsfrom f$.
\end{lemma}
\begin{proof}
  We first check that
  $\lam (b)$ is correctly typed
  \[
    \lam (b) \gg \tp = \Pi_{\sigma \gg A} (\tilde \sigma \gg B)
    = \sigma \gg \Pi_A B
  \]
  Conversely, since $f$ is of the type
  \[f \gg \tp = \sigma \gg \Pi_A B = \Pi_{\sigma \gg A} (\tilde \sigma \gg B)\]
  we have $\unlam (f) \gg \tp = \tilde \sigma \gg B$.
  It follows from the $\beta$- and $\eta$-rules that these are inverses of each other.
\end{proof}

\begin{definition}[$\Sigma$-types, \done{https://github.com/sinhp/HoTTLean/blob/TYPES2026/HoTTLean/Model/Unstructured/UnstructuredUniverse.lean\#L219}]
  An \emph{elementary $\Sigma$-type structure}
  on a universe $\tp : \Tm \to \Ty$
  consists of the following
  \begin{enumerate}
  \item 
    For any context $\Gamma$, type $A : \Gamma \to \Ty$,
    and dependent type $B : \Gamma . A \to \Ty$
    there is a type $\Sigma_A B : \Gamma \to \Ty$.
  \item
    The construction $\Sigma_A B$ is stable under substitution,
    meaning that for any $\sigma : \Delta \to \Gamma$ we have
    \[
      \Sigma_{\sigma \gg A} (\tilde{\sigma} \gg B) =
      \sigma \gg \Sigma_{A} B
    \]
    where $\tilde{\sigma}$ is the same as in \Cref{def:elementary-pi}.
  \item
    Under the assumptions of (1),
    for any $a : \Gamma \to \Tm$ and $b:\Gamma\to \Tm$
    such that $a \gg \tp = A$ and $b \gg \tp = \id_\Gamma . a \gg B$,
    there is a map $\pair (a,b) : \Gamma \to \Tm$
    satisfying $\pair (a,b) \gg \tp = \Sigma_A B$. 
    \[
\begin{tikzcd}
	\Gamma \\
	& {\Gamma . A} & \Tm \\
	& \Gamma & \Ty
	\arrow["{\id_\Gamma . a}"{description}, from=1-1, to=2-2]
	\arrow["a", bend left, from=1-1, to=2-3]
	\arrow["{\id_\Gamma}"', bend right, from=1-1, to=3-2]
	\arrow[from=2-2, to=2-3]
	\arrow[from=2-2, to=3-2]
	\arrow["\tp", from=2-3, to=3-3]
	\arrow["A"', from=3-2, to=3-3]
\end{tikzcd}
    \]
  \item
    The construction $\pair (a, b)$ is stable under substitution.
    For any $\sigma : \Delta \to \Gamma$
    \[
      \pair (\sigma \gg a,\sigma \gg b) =
      \sigma \gg \pair(a , b)
    \]
  \item
    Under the assumptions of (1),
    for any $s : \Gamma \to \Tm$
    such that $s \gg \tp = \Sigma_A B$,
    there are maps $\fst \, s : \Gamma  \to \Tm$
    and $\snd \, s : \Gamma  \to \Tm$
    satisfying $\fst\; s\gg \tp = A$
    and $\snd \; s\gg \tp = (\id_\Gamma . \fst\; s) \gg B$.
  \item ($\beta$-rules) The operations 
    $\fst$ and $\snd$ deconstruct a term formed using $\pair$.
    \[ \fst \;(\pair(a, b)) = a \quad \text{ and } 
    \quad \snd \; (\pair (a, b)) = b\]
  \item ($\eta$-rule) Conversely, the operation $\pair$
    reconstructs the term split up by $\fst$ and $\snd$.
  \[\pair (\fst \;s, \snd \;s) = s\]
  \end{enumerate}
\end{definition}

\begin{lemma}[\done{https://github.com/sinhp/HoTTLean/blob/TYPES2026/HoTTLean/Model/Unstructured/UnstructuredUniverse.lean\#L318},
\done{https://github.com/sinhp/HoTTLean/blob/TYPES2026/HoTTLean/Model/Unstructured/UnstructuredUniverse.lean\#L324}]
Stability of $\fst$ and $\snd$
under substitution can be derived from stability of $\pair$ under substitution
and the $\beta$- and $\eta$-rules.
\end{lemma}
\begin{proof}
\begin{align*}
  & \fst (\sigma \gg s) \\
  = \, & \fst (\sigma \gg \pair (\fst \; s, \snd \;s)) \\
  = \, & \fst (\pair (\sigma \gg \fst \; s, {\sigma} \gg \snd \; s)) \\
  = \, & \sigma \gg \fst \; s \\
  \\
  & \snd (\sigma \gg s) \\
  = \, & \snd (\sigma \gg \pair (\fst \; s, \snd \;s)) \\
  = \, & \snd (\pair (\sigma \gg \fst \; s, {\sigma} \gg \snd \; s)) \\
  = \, & {\sigma} \gg \snd \; s
\end{align*}
\end{proof}

\begin{lemma}\label{lem:substituted-pair}
  Suppose $\tp$ has elementary $\Sigma$-types,
  and fix a map $\sigma : \Delta \to \Gamma$.
  There is a bijection between pairs of commutative squares
  \begin{equation}\label{eq:substituted-pair-lhs}
  \begin{tikzcd}
    \Delta & \Tm && \Delta & \Tm \\
    \Gamma & \Ty && {\Gamma . A} & \Ty
    \arrow["a", from=1-1, to=1-2]
    \arrow["\sigma"', from=1-1, to=2-1]
    \arrow["\tp", from=1-2, to=2-2]
    \arrow["b", from=1-4, to=1-5]
    \arrow["{\sigma . a}"', from=1-4, to=2-4]
    \arrow["\tp", from=1-5, to=2-5]
    \arrow["A"', from=2-1, to=2-2]
    \arrow["B"', from=2-4, to=2-5]
  \end{tikzcd}
  \end{equation}
  and squares
  \begin{equation}
  \begin{tikzcd}\label{eq:substituted-pair-rhs}
      \Delta && \Tm \\
      \Gamma && \Ty
      \arrow["p", from=1-1, to=1-3]
      \arrow["\sigma"', from=1-1, to=2-1]
      \arrow["\tp", from=1-3, to=2-3]
      \arrow["{\Sigma_A B}"', from=2-1, to=2-3]
  \end{tikzcd}
  \end{equation}
  given by applying $(a, b) \mapsto \pair(a,b)$ forwards,
  and $(\fst \, p, \snd \, p) \mapsfrom p$ backwards.
\end{lemma}
\begin{proof}
  Suppose \Cref{eq:substituted-pair-lhs}.
  Then $\pair (a,b)$ is a term of the
  substituted $\Sigma$-type
  $\sigma \gg \Sigma_A B = \Sigma_{\sigma \gg A} (\tilde{\sigma} \gg B)$
  \[
    \begin{tikzcd}
      \Delta && \Tm \\
      \Gamma && \Ty
      \arrow["{{\pair(a,b)}}", from=1-1, to=1-3]
      \arrow["\sigma"', from=1-1, to=2-1]
      \arrow["\tp", from=1-3, to=2-3]
      \arrow["{\Sigma_A B}"', from=2-1, to=2-3]
    \end{tikzcd}
  \]
  Conversely, given \Cref{eq:substituted-pair-rhs},
  $p$ is a term of the
  substituted $\Sigma$-type
  $\sigma \gg \Sigma_A B = \Sigma_{\sigma \gg A} (\tilde{\sigma} \gg B)$,
  then the first component is typed correctly $\fst(p) \gg \tp = \sigma \gg A$ and $\snd(p)$
  can be typed as follows:
  \[
    \snd(p) \gg \tp = (\id_\Delta . \fst(p)) \gg \tilde \sigma \gg B = (\sigma . \fst(p)) \gg B
  \]
\end{proof}

\begin{definition}[$\Id$-types, \done{https://github.com/sinhp/HoTTLean/blob/TYPES2026/HoTTLean/Model/Unstructured/UnstructuredUniverse.lean\#L439}]
  An \emph{elementary $\Id$-type structure}
  on a universe $\tp : \Tm \to \Ty$
  consists of the following
  \begin{enumerate}
  \item 
    For any context $\Gamma$, type $A : \Gamma \to \Ty$,
    and terms $a_0, a_1 : \Gamma \to \Tm$ such that
    $a_i \gg \tp = A$,
    there is a type $\Id_A (a_0,a_1) : \Gamma \to \Ty$.
  \item
    The construction $\Id_A (a_0,a_1)$ is stable under substitution.
  \item
    For any context $\Gamma$, type $A : \Gamma \to \Ty$,
    and term $a : \Gamma \to \Tm$ such that
    $a \gg \tp = A$,
    there is a term $\refl \, a : \Gamma \to \Tm$
    satisfying
    $\refl \, a \gg \tp = \Id_A (a,a)$.
  \item The construction $\refl$ is stable under substitution.
  \item
    Under the assumptions of (3)
    one can construct the context $\Gamma . (x : A) . \Id_A(a,x)$
    by taking context extension on $\Gamma . A$ for the type
    $\Id_A(a,x)$ given by (1),
    and the substitution
    \[\rho_a := \id_\Gamma . a . \refl \, a : \Gamma \to \Gamma . (x : A) . \Id_A(a,x)\]
    For any \emph{motive} $C : \Gamma .(x : A) . \Id_A(a,x) \to \Ty$
    -- that means $C$ is parametrized by a family of paths with a fixed starting point $a$ --
    and map $c_\refl : \Gamma \to \Tm$ satisfying
    \[ c_\refl \gg \tp = \rho_a \gg C\]
    there is a term
    $j (C, c_\refl) : \Gamma . (x : A) . \Id_A(a,x) \to \Tm$
    satisfying $j \gg \tp = C$ and
    $\rho_a \gg j (C, c_\refl) = c_\refl$
    \[\begin{tikzcd}
	   \Gamma & \Tm \\
	   {\Gamma . (x : A) . \Id_A(a,x)} & \Ty
	   \arrow["{c_\refl}", from=1-1, to=1-2]
	   \arrow["{\rho_a}"', from=1-1, to=2-1]
	   \arrow["\tp", from=1-2, to=2-2]
	   \arrow["{j(C,c_\refl)}"{description}, dashed, from=2-1, to=1-2]
	   \arrow["C"', from=2-1, to=2-2]
    \end{tikzcd}\]
  \item The construction $j$ is stable under substitution.
  \end{enumerate}
\end{definition}

We will see eventually
(by combining 
\Cref{lem:coherent-weak-pullback},
\Cref{prop:elementary-id-to-algebraic-id}
and \Cref{prop:algebraic-id-to-elementary-id})
that the last condition (6) can be obtained by an adjustment of $j$.
If $j$ is not stable under substitution,
we can always replace it with another $j'$ that is.

\begin{remark}
In the previous exposition of type formers,
we have exclusively worked with a single universe $\tp$,
for the sake of simplicity.
However,
we can generalize the definitions for $\Pi$-types and
$\Sigma$-types
to hold with $A$ in a universe $\tp_0$,
$B$ in a second universe $\tp_1$,
and $\Pi_A B$ in a third universe $\tp_2$.
For example, this more general formulation can
be used to describe Lean's ``universe polymorphic'' $\Pi$-types.
One can find the more general formulation in our formalization.

In the presence of multiple universes,
one typically requires that identity types are
\emph{large eliminating},
meaning that the motive for identity elimination $C$
need not live in the same universe as the type $A$ on which
identities are taken.
In some systems one might even require that $\Id_A (a_0,a_1)$ lives
in a different universe.
For example, in Lean's type theory $\Id_A (a_0,a_1)$
lives in a universe of propositions \textsf{Prop}.
\end{remark}

As we will see in \Cref{sec:StrSem},
working with a single universe can lead to \textsf{Type : Type} paradoxes,
which we resolve by considering universe lifts.
Typically, one might consider an infinite sequence
of universes lifts such that each type
-- including universes -- are terms of some large enough universe.
The following definition presents semantics for
non-cumulative universe lifts
``\`a la Coquand'' \cite{coquand2013, gratzer2020, favonia2023}.

\begin{definition}[\done{https://github.com/sinhp/HoTTLean/blob/TYPES2026/HoTTLean/Model/Unstructured/UHom.lean\#L23},
\done{https://github.com/sinhp/HoTTLean/blob/TYPES2026/HoTTLean/Model/Unstructured/UHom.lean\#L77}]
  A \emph{morphism of universes} $l : \tp_0 \to \tp_1$ 
  consists of a pair of maps $l_\Ty : \Ty_0 \to \Ty_1$
  and $l_\Tm : \Tm_0 \to \Tm_1$ such that 
  the following square is a pullback
  \[\begin{tikzcd}
	{\Tm_0} & {\Tm_1} \\
	{\Ty_0} & {\Ty_1}
	\arrow["{{{l_\Tm}}}", from=1-1, to=1-2]
	\arrow["{{{\tp_0}}}"', from=1-1, to=2-1]
	\arrow["\lrcorner"{anchor=center, pos=0.125}, draw=none, from=1-1, to=2-2]
	\arrow["{\tp_1}", from=1-2, to=2-2]
	\arrow["{{{l_\Ty}}}"', from=2-1, to=2-2]
\end{tikzcd}\]
  Suppose $\CC$ has a terminal object $1$.
  A \emph{universe lift} consists of
  a morphism of universes $l : \tp_0 \to \tp_1$
  and a pair of maps $U_0 : 1 \to \Ty_1$
  and $\mathsf{asTm} : \Ty_0 \to \Tm_1$
  such that the following square is a pullback
\[\begin{tikzcd}
	{\Ty_0} & {\Tm_1} \\
	1 & {\Ty_1}
	\arrow["{{\mathsf{asTm}}}", from=1-1, to=1-2]
	\arrow[from=1-1, to=2-1]
	\arrow["\lrcorner"{anchor=center, pos=0.125}, draw=none, from=1-1, to=2-2]
	\arrow["{{{\tp_1}}}", from=1-2, to=2-2]
	\arrow["{{{U_0}}}"', from=2-1, to=2-2]
\end{tikzcd}\]
\end{definition}
The maps $l_\Ty$ and $l_\Tm$ allow types and terms
from the smaller universe
to be lifted to the larger one.
The first pullback square ensures that universe lifts commute 
up to isomorphism with context extension.
\[ \Gamma . A \iso \Gamma . (A \gg l_\Ty) \]
The second pullback square provides a type in the larger universe
that represents the smaller universe,
and is equivalent to providing an isomorphism
\[ 1.U_0 \iso \Ty_0 \]
One could also place additional requirements on
morphisms of universes,
such as requiring that $l_\Ty$ is monomorphic,
or that the above isomorphisms are strict equalities
$\Gamma . A = \Gamma . (A \gg l_\Ty)$ and $1.U_0 = \Ty_0$.

\section{Polynomial functors in \foreignlanguage{greek}{π}-clans}\label{sec:Poly}

If $\CC$ is LCC,
any map $f:E\to B$ in $\CC$
can be considered a \emph{polynomial signature}.
Each such polynomial signature induces a \emph{polynomial functor},
denoted as $P_f:\CC\to \CC$, as the composition
\[
\begin{tikzcd}
    \CC \ar[r,"E^*"]& \CC/E\ar[r,"f_*"] & \CC/B\ar[r,"B_!"] & \CC
\end{tikzcd}
\]

In this section, we show that $\pi$-clans form
a good setting for developing a general theory of polynomial functors.
In this setting, we also define a polynomial functor as a three-fold
composition,
but restricted to suitable subcategories.

\begin{remark}\label{rmk:univariate}
  To simplify the exposition,
  we are only presenting polynomial functors
  ``in a single variable''.
  In the Lean formalization,
  we first build polynomial functors ``in many variables''
  and then specialize to the single variable case.
  One can find a presentation of polynomial functors
  ``in many variables'' in \cite{gambino2013}.
\end{remark}

\begin{definition}[\done{https://github.com/sinhp/HoTTLean/blob/TYPES2026/HoTTLean/ForMathlib/CategoryTheory/Adjunction/PartialAdjoint.lean\#L16}]
  Fix a functor $L : \CC \to \DD$.
  A partial right adjoint of $L$,
  denoted $L \dashv_\partial \partial R$,
  consists of two full subcategories
  $\partial\CC \hookrightarrow \CC$
  and $\partial\DD \hookrightarrow \DD$,
  a functor $\partial R : \partial \DD \to \partial \CC$, as depicted below,
  as well as a hom-set bijection for all
  $X \in \CC^\op$ and $Y \in \partial\DD$
      \[ \DD(L (X), Y) \iso \CC(X,\partial R (Y))\]
  natural in both $X$ and $Y$.
  \[\begin{tikzcd}
	\CC & {\partial\CC} \\
	\DD & {\partial \DD}
	\arrow[""{name=0, anchor=center, inner sep=0}, "L"', from=1-1, to=2-1]
	\arrow[hook', from=1-2, to=1-1]
	\arrow[""{name=1, anchor=center, inner sep=0}, "{\partial R}"', from=2-2, to=1-2]
	\arrow[hook', from=2-2, to=2-1]
	\arrow["{\dashv_\partial}"{description}, draw=none, from=0, to=1]
\end{tikzcd}\]
\end{definition}

The condition of being a partial right adjoint
is stronger than being a right adjoint of a restriction of $L$,
\[\partial L : \partial \CC \to \partial \DD\]
when such a restriction exists.
Indeed, consider the right-hand side of the isomorphism, whereas the domain of the map in $\CC$ is only allowed to be in $\partial C$ for the restricted adjunction, we allow it to be any object in $C$.

The following definition is adapted from Joyal's notes \cite{joyal2017},
and can be traced further back to \cite{taylor1987}.
The combination of conditions does not appear as a 
single definition in Lean,
to avoid creating redundant hypotheses.
Links to the formalization are therefore provided for each individual condition.
\begin{definition} \label{def:clan}
  A \emph{preclan}\footnote{
    We have introduced the terminology ``preclan'' for convenience,
    and it does not appear in the literature.}
  $(\CC,\RR)$ consists of a category $\CC$ with a class of maps $\RR$ in $\CC$
  whose members will be called \emph{$\RR$-maps},
  satisfying the following
  \begin{enumerate}
    \item Pullbacks of $\RR$-maps along all morphisms exist and are $\RR$-maps. (\mathlib{https://leanprover-community.github.io/mathlib4_docs/Mathlib/CategoryTheory/MorphismProperty/Limits.html\#CategoryTheory.MorphismProperty.IsStableUnderBaseChangeAlong},
    \mathlib{https://leanprover-community.github.io/mathlib4_docs/Mathlib/CategoryTheory/MorphismProperty/Limits.html\#CategoryTheory.MorphismProperty.HasPullbacksAlong})
    \item All isomorphisms are in $\RR$. (\mathlib{https://leanprover-community.github.io/mathlib4_docs/Mathlib/CategoryTheory/MorphismProperty/Composition.html\#CategoryTheory.MorphismProperty.isomorphisms_le_of_containsIdentities})
    \item $\RR$ is closed under composition. (\mathlib{https://leanprover-community.github.io/mathlib4_docs/Mathlib/CategoryTheory/MorphismProperty/Composition.html\#CategoryTheory.MorphismProperty.IsStableUnderComposition})
  \end{enumerate}
  
  Following Joyal,
  we use $\RR(X)$ to denote the full subcategory of the slice
  $\CC / X$
  consisting of objects with an underlying map in $\RR$.
  For any $f : Y \to X$ we have a functor $f_! : \CC / Y \to \CC / X$ between the slice categories,
  and a partial right adjoint
  \[\begin{tikzcd}
	{\CC / Y} & {\RR(Y)} \\
	{\CC / X} & {\RR(X)}
	\arrow[""{name=0, anchor=center, inner sep=0}, "{f_!}"', from=1-1, to=2-1]
	\arrow[hook', from=1-2, to=1-1]
	\arrow[""{name=1, anchor=center, inner sep=0}, "{\partial f^*}"', from=2-2, to=1-2]
	\arrow[hook', from=2-2, to=2-1]
	\arrow["{\dashv_\partial}"{description}, draw=none, from=0, to=1]
\end{tikzcd}\]
  where $\partial f^*$ takes an $\RR$-map over $X$
  to its pullback along $f$.
  If $f$ is an $\RR$-map then $\partial f^*$
  extends to a pullback functor $f^* : \CC / X \to \CC / Y$
  between the slice categories,
  and $f_!$ restricts to a map
  $\partial f_! : \RR(Y) \to \RR(X)$.
  Furthermore,
  \[f_! : \CC / Y \dashv \CC / X : f^* \quad \text{and} \quad
  \partial f_! : \RR(Y) \dashv \RR(X) : \partial f^*\]

  A preclan $(\CC,\RR)$ will be called a
  \emph{$\pi$-preclan} when additionally
  \begin{enumerate}[resume]
    \item $\RR$ is closed under pushforward. (\done{https://github.com/sinhp/HoTTLean/blob/TYPES2026/HoTTLean/ForMathlib/CategoryTheory/MorphismProperty/OverAdjunction.lean\#L222},
      \done{https://github.com/sinhp/HoTTLean/blob/TYPES2026/HoTTLean/ForMathlib/CategoryTheory/MorphismProperty/OverAdjunction.lean\#L229}).
      This means that for any $\RR$-map $f : Y \to X$
      there is a partial right adjoint to pullback
      \[
    \begin{tikzcd}
	{\CC / X} & {\RR(X)} \\
	{\CC / Y} & {\RR(Y)}
	\arrow[""{name=0, anchor=center, inner sep=0}, "{f^*}"', from=1-1, to=2-1]
	\arrow[hook', from=1-2, to=1-1]
	\arrow[""{name=1, anchor=center, inner sep=0}, "{\partial f_*}"', from=2-2, to=1-2]
	\arrow[hook', from=2-2, to=2-1]
	\arrow["{\dashv_\partial}"{description}, draw=none, from=0, to=1]
    \end{tikzcd}
      \]
    
  \end{enumerate}
  It follows that for $f : Y \to X$ in a $\pi$-preclan $\RR$,
  we have an adjunction
  \[\partial f^* : \RR(X) \dashv \RR(Y) : \partial f_*\]
  
  Suppose $\CC$ has a terminal object $1$.
  We will refer to objects $X$ such that the unique map $X \to 1$
  is in $\RR$ as \emph{$\RR$-objects}.
  A ($\pi$-)preclan is called a ($\pi$-)clan when additionally
  \begin{enumerate}[resume]
    \item All objects are $\RR$-objects. (\done{https://github.com/sinhp/HoTTLean/blob/TYPES2026/HoTTLean/ForMathlib/CategoryTheory/MorphismProperty/Limits.lean\#L75})
  \end{enumerate}

  Following Joyal \cite{joyal2017},
  for a preclan $(\CC,\RR)$
  we will call $(\RR(X),\RR |_{\RR(X)})$
  the \emph{local clan} at $X$.
\end{definition}

\begin{example}
We provide some examples of the foregoing definitions.
\begin{itemize}
  \item If $\CC$ is an LCC category, then (4) can be phrased as
    $\RR$ being closed under the ambient pushforward operation.
    In particular, if $\RR$ consists of all the maps in $\CC$,
    then $(\CC,\RR)$ forms a $\pi$-clan.
  \item If $\CC$ is a Quillen model category and $\RR$ is
    the class of fibrations,
    then $(\CC,\RR)$ forms a preclan.
    If additionally $\CC$ is LCC and the model structure satisfies
    the Frobenius property,
    then $(\CC,\RR)$ is a $\pi$-preclan \cite{gambino2017}.
    Of course, in either case
    restricting to the full subcategory of fibrant objects
    results in a ($\pi$-)clan.
  \item It follows from the previous example that the
    classical model structure on the category of small categories
    forms a clan $(\Cat,\II)$,
    where $\II$ consists of (split)
    isofibrations \cite{joyal2017}.
    This restricts to a $\pi$-clan on the category
    of small groupoids $(\Grpd,\II |_\Grpd)$ \cite{joyal2017}. 
    That $(\Grpd,\II |_\Grpd)$ is a $\pi$-clan
    has been formalized in Lean
    (\done{https://github.com/sinhp/HoTTLean/blob/TYPES2026/HoTTLean/Groupoids/SplitIsofibration.lean\#L42},
    \done{https://github.com/sinhp/HoTTLean/blob/TYPES2026/HoTTLean/Groupoids/SplitIsofibration.lean\#L47},
    \done{https://github.com/sinhp/HoTTLean/blob/TYPES2026/HoTTLean/Groupoids/SplitIsofibration.lean\#L69},
    \done{https://github.com/sinhp/HoTTLean/blob/TYPES2026/HoTTLean/Groupoids/SplitIsofibration.lean\#L293}).
  \item The Kan-Quillen model structure on $\sSet$
    produces an $\pi$-preclan \cite{joyal2017}.
  \item A model of MLTT with sufficient structure
    (enough universes, $\Unit$-, $\Sigma$-, and $\Pi$-types)
    forms a $\pi$-clan.
    A precise statement of this fact is given in
    \Cref{cor:elementary-universe-hierarchy-to-algebraic}.
\end{itemize}
\end{example}

Henceforth we fix a $\pi$-clan $(\CC,\RR)$ and an $\RR$-map $f : E \to B$.

\begin{definition}[\done{https://github.com/sinhp/HoTTLean/blob/TYPES2026/HoTTLean/ForMathlib/CategoryTheory/Polynomial.lean\#L649},
\done{https://github.com/sinhp/HoTTLean/blob/TYPES2026/HoTTLean/ForMathlib/CategoryTheory/Polynomial.lean\#L701}]
  \label{def:poly}
  The \emph{polynomial functor} $P_f : \CC \to \CC$
  associated with the
  \emph{polynomial signature} $f \in \RR$ is defined by the following composition
\[\begin{tikzcd}
	\CC & {\RR(1)} & {\RR(E)} & {\RR(B)} & {\RR(1)} & \CC
	\arrow["\simeq"{description}, draw=none, from=1-1, to=1-2]
	\arrow["{Y^*}", from=1-2, to=1-3]
	\arrow["{f_*}", from=1-3, to=1-4]
	\arrow["{X_!}", from=1-4, to=1-5]
	\arrow["\simeq"{description}, draw=none, from=1-5, to=1-6]
\end{tikzcd}\]
\end{definition}

\begin{proposition}[\done{https://github.com/sinhp/HoTTLean/blob/TYPES2026/HoTTLean/ForMathlib/CategoryTheory/Polynomial.lean\#L915},
\done{https://github.com/sinhp/HoTTLean/blob/TYPES2026/HoTTLean/ForMathlib/CategoryTheory/Polynomial.lean\#L929},
\done{https://github.com/sinhp/HoTTLean/blob/TYPES2026/HoTTLean/ForMathlib/CategoryTheory/Polynomial.lean\#L946}] \label{prop:poly-up}
  A polynomial functor is characterized by its universal property:
  there is a bijection
  \[ \CC(\Gamma, P_f X) \iso
    \sum_{\fst \in \CC(\Gamma,B)}\CC(\fst \times_B f, X)\]
  that is natural in both $\Gamma \in \CC^\op$ and $X \in \CC$.
  For a given $t : \Gamma \to P_f X$,
  an element of the dependent sum can be visualized as
\[\begin{tikzcd}
	X & {\fst (t) \times_B f} & E \\
	& \Gamma & B
	\arrow["\snd (t)"', from=1-2, to=1-1]
	\arrow[from=1-2, to=1-3]
	\arrow[from=1-2, to=2-2]
	\arrow["\lrcorner"{anchor=center, pos=0.125}, draw=none, from=1-2, to=2-3]
	\arrow["f", from=1-3, to=2-3]
	\arrow["\fst (t)"', from=2-2, to=2-3]
\end{tikzcd}\]
  Of particular interest is the natural bijection applied to
  the identity $\id \in \CC(P_f X, P_f X)$,
  which results in a {universal} pair
  $\fstProj = \fst (\id) : P_f X \to B$
  and $\sndProj = \snd (\id) : \fstProj \times_B f \to X$,
  through which all others factor
\[\begin{tikzcd}[column sep = large]
	& X \\
	{\fst(t) \times_B f} & {\fstProj \times_B f} & E \\
	\Gamma & {P_f X} & B
	\arrow["{\snd(t)}", from=2-1, to=1-2]
  \arrow["{t \times_B f}"', from=2-1, to=2-2]
	\arrow[from=2-1, to=3-1]
	\arrow["\lrcorner"{anchor=center, pos=0.125}, draw=none, from=2-1, to=3-2]
	\arrow["{\sndProj}"', from=2-2, to=1-2]
	\arrow[from=2-2, to=2-3]
	\arrow[from=2-2, to=3-2]
	\arrow["\lrcorner"{anchor=center, pos=0.125}, draw=none, from=2-2, to=3-3]
	\arrow["f", from=2-3, to=3-3]
	\arrow["t"', from=3-1, to=3-2]
	\arrow["{\fst(t)}"', bend right, from=3-1, to=3-3]
	\arrow["{\fstProj}"', from=3-2, to=3-3]
\end{tikzcd}\]
\end{proposition}
\begin{proof}
  We first define $\fstProj : P_f X \to B$
  as the underlying morphism defining the object $f_* E^* X$ in the local clan $\RR(B)$.
  Then for $t : \Gamma \to P_f X$ we can define $\fst(t) := t \gg \fstProj : \Gamma \to B$.
  Note that indeed $\fstProj = \fst (\id)$.

  By combining the (partial) adjunctions $f^* \dashv f_*$ and $E_! \dashv E^*$ we have
  \begin{align*}
    & \CC / B (\fst(t), \fstProj) \\
    = &\, \CC / B (\fst(t), f_* E^* X) \\
    \iso &\, \CC / E (f^* \fst(t), E^* X) \\
    \iso &\, \CC (E_! f^* \fst(t), X) \\
    = &\, \CC (\fst(t) \times_B f, X)
  \end{align*}
  Note that the full strength of partial adjunction, rather than merely adjunction restricted to subcategories, is used. Indeed, as $\fst(t)$ (and consequently, $f^*\fst(t)$ and $E_!f^*\fst(t)$) does not need to be in $\RR$, all three equations in the middle of the derivation require the isomorphism from partial adjunctions to be derived.
  We use this to define $\snd(t) : \fst(t) \times_B f \to X$
  as the transpose of $t : \fst(t) \to \fstProj$ in the slice $\CC / B$ along the sequence of isomorphisms above.
  Then we can define $\sndProj := \snd(\id)$,
  and verify using the naturality of the (partial) adjunction that
  \[ \snd(t) = (t \times_B f) \gg \sndProj\]
  We leave it to the reader to reconstruct $t$ from $\fst(t)$ and $\snd(t)$.
  
  To compute $\sndProj$ (and therefore $\snd$) explicitly ,
  we provide the following equation:
  \[
    \begin{tikzcd}[column sep = large]
    {E_! f^* f_* E^* X} & {E_! E^* X} \\
    {\fstProj \times_B f} & {E \times X} \\
    & X
    \arrow["{E_! \counit_{E^*X}}", from=1-1, to=1-2]
    \arrow[equals, from=1-1, to=2-1]
    \arrow[equals, from=1-2, to=2-2]
    \arrow["{\counit_X}", bend left = 60, from=1-2, to=3-2]
    \arrow["{\epsilon_X}", from=2-1, to=2-2]
    \arrow["\sndProj"', from=2-1, to=3-2]
    \arrow["{E^* X}"', from=2-2, to=3-2]
  \end{tikzcd}
  \]
  where the
  $E^* X$ component of the counit of the adjunction $f^* \dashv f_*$ is denoted by $\epsilon_X$.  
\end{proof}

The structure of a $\pi$-clan is enough to make the composition of polynomial functors work correctly. The following definition is in a precise analogy with its counterpart in the LCC setting. 

\begin{definition}[\done{https://github.com/sinhp/HoTTLean/blob/TYPES2026/HoTTLean/ForMathlib/CategoryTheory/Polynomial.lean\#L882}]\label{def:pcomp}
  If $f : E \to B$ and $f' : E' \to B'$ are both in $\RR$
  then we can form the \emph{polynomial composition}
  $f' \pcomp f : \compDom \to P_f B'$ as follows
  \[
  \begin{tikzcd}
    {E'} & \compDom \\
    {B'} & {\fstProj \times_B f} & E \\
    & {P_fB'} & B
    \arrow["{f'}"', from=1-1, to=2-1]
    \arrow[from=1-2, to=1-1]
    \arrow["\lrcorner"{anchor=center, pos=0.125, rotate=-90}, draw=none, from=1-2, to=2-1]
    \arrow[from=1-2, to=2-2]
    \arrow["{f' \pcomp f}"{description, pos=0.2}, shift left=5, bend left = 50, from=1-2, to=3-2]
    \arrow["{\sndProj}", from=2-2, to=2-1]
    \arrow[from=2-2, to=2-3]
    \arrow[from=2-2, to=3-2]
    \arrow["\lrcorner"{anchor=center, pos=0.125}, draw=none, from=2-2, to=3-3]
    \arrow["f", from=2-3, to=3-3]
    \arrow["{\fstProj}"', from=3-2, to=3-3]
  \end{tikzcd}\]
  This is the polynomial signature for a polynomial functor
  since $\RR$ is stable under pullback (axiom (1)) and closed under composition
  (axiom (3)).
  The associated polynomial functor $P_{f' \pcomp f} : \CC \to \CC$
  is characterized by the natural isomorphism\footnote{
    This isomorphism is not needed in the remaining development
    but is mentioned for completeness. Its construction is, therefore, omitted.}
  \[ P_{f' \pcomp f} \iso P_{f'} \gg P_f\]
\end{definition}

\begin{lemma}[\done{https://github.com/sinhp/HoTTLean/blob/TYPES2026/HoTTLean/ForMathlib/CategoryTheory/Clan.lean\#L161},
\done{https://github.com/sinhp/HoTTLean/blob/TYPES2026/HoTTLean/ForMathlib/CategoryTheory/Clan.lean\#L473}]
  \label{lem:precomposition-BC}
  For any commutative square
\[\begin{tikzcd}
	{E'} & {B'} \\
	E & B
	\arrow["{{{f'}}}", from=1-1, to=1-2]
	\arrow["\phi", from=2-1, to=1-1]
	\arrow["f"', from=2-1, to=2-2]
	\arrow["\delta"', from=2-2, to=1-2]
\end{tikzcd}
\]
  If $f$ and $f'$ are both $\RR$-maps,
  then there are Beck--Chevalley (BC) natural transformations
  \[\begin{tikzcd}
	{\RR(E')} & {\RR(B')} \\
	{\RR(E)} & {\RR(B)}
	\arrow["{{f'_!}}", from=1-1, to=1-2]
	\arrow["{{\phi^*}}"', from=1-1, to=2-1]
	\arrow["{{\delta^*}}", from=1-2, to=2-2]
	\arrow["{{f_!}}"', from=2-1, to=2-2]
	\arrow["\alpha"{description}, Rightarrow, from=2-1, to=1-2]
\end{tikzcd}
  \quad
  \quad
  \begin{tikzcd}
	{\RR(E')} & {\RR(B')} \\
	{\RR(E)} & {\RR(B)}
	\arrow["{{f'_*}}", from=1-1, to=1-2]
	\arrow["{{\phi^*}}"', from=1-1, to=2-1]
	\arrow["{{\delta^*}}", from=1-2, to=2-2]
	\arrow["{{f_*}}"', from=2-1, to=2-2]
	\arrow["\beta"{description}, Rightarrow, from=1-2, to=2-1]
\end{tikzcd}\]
  Furthermore, if the square is a pullback
  then these are natural isomorphisms.
\end{lemma}
\begin{proof}
  These are respectively
  defined as the left and right \emph{mates}
  (see
  \href{https://math.iisc.ac.in/~gadgil/PfsProgs25doc/Mathlib/CategoryTheory/Adjunction/Mates.html}{Mathlib}
  and \cite{hazratpour2024})
  of the following natural isomorphisms (which are inverses of one another)
\[
\begin{tikzcd}
	{\RR(E')} & {\RR(B')} \\
	{\RR(E)} & {\RR(B)}
	\arrow["{{{{\phi^*}}}}"', from=1-1, to=2-1]
	\arrow["{{{{f'^*}}}}"', from=1-2, to=1-1]
	\arrow["\Rightarrow"{marking, allow upside down}, draw=none, from=1-2, to=2-1]
	\arrow["{{{{\delta^*}}}}", from=1-2, to=2-2]
	\arrow["{{{{f^*}}}}"', from=2-1, to=2-2]
\end{tikzcd}
\quad
\quad
\begin{tikzcd}
	{\RR(E')} & {\RR(B')} \\
	{\RR(E)} & {\RR(B)}
	\arrow["{{{{\phi^*}}}}"', from=1-1, to=2-1]
	\arrow["{{{{f'^*}}}}"', from=1-2, to=1-1]
	\arrow["{{{{\delta^*}}}}", from=1-2, to=2-2]
	\arrow["\Rightarrow"{marking, allow upside down}, draw=none, from=2-1, to=1-2]
	\arrow["{{{{f^*}}}}"', from=2-1, to=2-2]
\end{tikzcd}
\]
The usual approach to such a proof in an LCC category relies on the theory of \emph{conjugate}, as in \cite{hazratpour2024}. In our case, we work with mate calculus instead.


For a morphism $f$ in a $\pi$-clan $\RR$,
one can verify that
$f_!, f_* : \RR(E) \to \RR(B)$
and $f^* : \RR(B) \to \RR(E)$ all commute (up to canonical natural isomorphism)
with Yoneda embeddings
$y : \RR(X) \to \Psh(\CC) / y(X)$. For example, the diagram for $f_*$ is 
\[
\begin{tikzcd}
    {\mathcal R}(E)\ar[r,"{f_*}"]\ar[d,hook,"y"] &  {\mathcal R}(B) \ar[d,hook,"y"]\\
    \Psh(\mathcal C)/y(E) \ar[r,"{f_*}"]& \Psh(\mathcal C)/y(B)
\end{tikzcd}
\]

It follows that $\alpha$ and $\beta$
are the restrictions of the corresponding mates
$\tilde \alpha$ and $\tilde \beta$ in $\Psh(\CC)$.
\[\begin{tikzcd}
	{\Psh(\CC) /y(E')} & {\Psh(\CC) / y(B')} \\
	{\Psh(\CC) / y(E)} & {\Psh(\CC) / y(B)}
	\arrow["{{y(f')_!}}", from=1-1, to=1-2]
	\arrow["{{y(\phi)^*}}"', from=1-1, to=2-1]
	\arrow["{{y(\delta)^*}}", from=1-2, to=2-2]
	\arrow["{{y(f)_!}}"', from=2-1, to=2-2]
	\arrow["\tilde \alpha"{description}, Rightarrow, from=2-1, to=1-2]
\end{tikzcd}
  \quad
  \quad
  \begin{tikzcd}
	{\Psh(\CC) /y(E')} & {\Psh(\CC) / y(B')} \\
	{\Psh(\CC) / y(E)} & {\Psh(\CC) / y(B)}
	\arrow["{{y(f')_*}}", from=1-1, to=1-2]
	\arrow["{{y(\phi)^*}}"', from=1-1, to=2-1]
	\arrow["{{y(\delta)^*}}", from=1-2, to=2-2]
	\arrow["{{y(f)_*}}"', from=2-1, to=2-2]
	\arrow["\tilde \beta"{description}, Rightarrow, from=1-2, to=2-1]
\end{tikzcd}\]
Since $\Psh(\CC)$ is LCC, the general theory of the calculus \cite{hazratpour2024} of conjugates applies, which gives that 
$\tilde \alpha$ and $\tilde \beta$ are isomorphisms
when the commutative square $y(\phi) \gg y(f') = y(f) \gg y(\delta)$ is a pullback.
\end{proof}

\begin{proposition} \label{prop:vertical-nat-trans-computation}
  For any triangle
\[\begin{tikzcd}
	{E} && E' \\
	& B
	\arrow["\rho", from=1-1, to=1-3]
	\arrow["{f}"', from=1-1, to=2-2]
	\arrow["f'", from=1-3, to=2-2]
\end{tikzcd}\]
  If $f$ and $f'$ are $\RR$-maps,
  there is a natural transformation\footnote{
    The definition of $v$ is formalized in HoTTLean (\done{https://github.com/sinhp/HoTTLean/blob/TYPES2026/HoTTLean/ForMathlib/CategoryTheory/Polynomial.lean\#L465}),
    but the computation rules for its components are not.
  }
  $v : P_{f'} \to P_{f}$ given by the whiskering diagram
\[\begin{tikzcd}
	\CC & {\RR(E')} & {\RR(B)} \\
	& {\RR(E)} & {\RR(B)} & \CC
	\arrow["{E'^*}", from=1-1, to=1-2]
	\arrow[""{name=0, anchor=center, inner sep=0}, "{E^*}"', from=1-1, to=2-2]
	\arrow["{f'_*}", from=1-2, to=1-3]
	\arrow["{\rho^*}", from=1-2, to=2-2]
	\arrow["\Rightarrow"{marking, allow upside down}, draw=none, from=1-3, to=2-2]
	\arrow["{\id^*}"', equals, from=1-3, to=2-3]
	\arrow[""{name=1, anchor=center, inner sep=0}, "{B_!}", from=1-3, to=2-4]
	\arrow["{f_*}"', from=2-2, to=2-3]
	\arrow["{B_!}"', from=2-3, to=2-4]
	\arrow["\iso"{marking, allow upside down}, draw=none, from=0, to=1-2]    
	\arrow["{=}"{marking, allow upside down}, draw=none, from=1, to=2-3]
\end{tikzcd}\]
  where the middle cell is given by BC
  (\Cref{lem:precomposition-BC}).

  Furthermore,
  on components $v_X$ can be computed in relation to the universal property
  in \Cref{prop:poly-up} as follows.
  For a map $t : \Gamma \to P_{f'} X$,
  the projections satisfy
  \[\fst (t \gg v_X) = \fst (t)
  \quad
  \quad
  \snd(t \gg v_X) = (\fst(t) \times_B \rho) \gg \snd(t)
  \]
  where $\fst(t) \times_B \rho$ is depicted in the following diagram
  \[\begin{tikzcd}
	& {\fst(t) \times_B f} & E \\
	X & {\fst(t) \times_B f'} & {E'} \\
	& \Gamma & B
	\arrow[from=1-2, to=1-3]
	\arrow["{\snd(t \gg v_X)}"', from=1-2, to=2-1]
	\arrow["{\fst(t) \times_B \rho}", from=1-2, to=2-2]
	\arrow["\rho", from=1-3, to=2-3]
	\arrow["{\snd(t)}", from=2-2, to=2-1]
	\arrow[from=2-2, to=2-3]
	\arrow[from=2-2, to=3-2]
	\arrow["\lrcorner"{anchor=center, pos=0.125}, draw=none, from=2-2, to=3-3]
	\arrow["{f'}", from=2-3, to=3-3]
	\arrow["{\fst(t)}"', from=3-2, to=3-3]
\end{tikzcd}\]
\end{proposition}
\begin{proof}
  Let us write $\fstProj' = f_* E^* X$ for the polynomial functor applied to $X$
  as an object in the local clan $\RR(B)$ (as opposed to an object in $\CC$).
  By unfolding the defintions of $v$ and the BC natural transformations,
  one can explicitly write out the component at $X$ as $v_X = B_! (\nu_X)$ where
  $\nu_X : \fstProj' \to \fstProj$ is the following composition in the local clan $\RR(B)$
  \begin{equation} \label{eq:vertical-nat-trans-computation-lemma-0}
  \begin{tikzcd}[column sep = large]
    {\fstProj'} && \fstProj \\
    {f_* f^*\fstProj'} & {f_*\rho^* f'^* \fstProj'} & {f_* \rho^* E'^* X}
    \arrow["{\nu_X}", from=1-1, to=1-3]
    \arrow["{\unit_{\fstProj'}}"', from=1-1, to=2-1]
    \arrow["{f_* n_{\fstProj'}}"', from=2-1, to=2-2]
    \arrow["\iso"{description}, shift left=3, draw=none, from=2-1, to=2-2]
    \arrow["{f_* \rho^* \counit_{E'^* X}}"', from=2-2, to=2-3]
    \arrow["{f_* m_X^{-1}}"', from=2-3, to=1-3]
    \arrow["\iso"{description}, shift left=3, draw=none, from=2-3, to=1-3]
  \end{tikzcd}
  \end{equation}
  Here $n : f^* \iso \rho^* f'^*$ and $m : E^* \iso \rho^* E'^*$
  are the canonical natural isomorphisms due to commutativity,
  $\unit : \id_{\RR(B)} \to f_* f^*$ is the unit of the pullback-pushforward adjunction for $f$,
  and $\counit : f'^* f'_* \to \id_{\RR(E')}$ is the counit of the pullback-pushforward adjunction for $f'$.

  By definition of $v_X = B_! (\nu_X)$ being a morphism in the local clan $\RR(X)$,
  the following triangle commutes
  \begin{equation} \label{eq:fstProj}
  \begin{tikzcd}
    {P_{f'} X} & {P_f X} \\
    & B
    \arrow["{v_X}", from=1-1, to=1-2]
    \arrow["{\fstProj'}"', from=1-1, to=2-2]
    \arrow["\fstProj", from=1-2, to=2-2]
  \end{tikzcd}
  \end{equation}
  We would also like to establish the following equation relating
  $\sndProj : \fstProj \times_B f$,
  and $\sndProj' : \fstProj' \times_B f'$.
  \begin{equation} \label{eq:sndProj}
  \begin{tikzcd}
    {\fstProj' \times_B f} & {\fstProj \times_B f} \\
    {\fstProj' \times_B f'} & X
    \arrow["{v_X \times_B f}", from=1-1, to=1-2]
    \arrow["{\fstProj' \times_B \rho}"', from=1-1, to=2-1]
    \arrow["\sndProj", from=1-2, to=2-2]
    \arrow["{\sndProj'}"', from=2-1, to=2-2]
  \end{tikzcd}
  \end{equation}
  which requires a careful proof.
  Before proving \Cref{eq:sndProj},
  let us use these two equations to show the desired result
  about $\fst(t \gg v_X)$ and $\snd(t \gg v_X)$.
  Using \Cref{eq:fstProj},
  \[ \fst(t \gg v_X) = t \gg v_X \gg \fstProj = t \gg \fstProj' = \fst(t) \]
  Using \Cref{eq:sndProj}, \Cref{prop:poly-up}, and the following diagram
  \[
  \begin{tikzcd}
	{\fst(t) \times_B f} & {\fstProj' \times_B f} & E \\
	{\fst(t) \times_B f'} & {\fstProj' \times_B f'} & {E'} \\
	\Gamma & {P_{f'} X} & B
	\arrow["{t \times_B f}", from=1-1, to=1-2]
	\arrow["{\fst(t) \times_B \rho}"', from=1-1, to=2-1]
	\arrow[from=1-2, to=1-3]
	\arrow["{\fstProj' \times_B \rho}"', from=1-2, to=2-2]
	\arrow["\rho"', from=1-3, to=2-3]
	\arrow["f", bend left, from=1-3, to=3-3]
	\arrow["{t \times_B f'}"', from=2-1, to=2-2]
	\arrow[from=2-1, to=3-1]
	\arrow[from=2-2, to=2-3]
	\arrow[from=2-2, to=3-2]
	\arrow["{f'}"', from=2-3, to=3-3]
	\arrow["t", from=3-1, to=3-2]
	\arrow["{\fst(t)}"', bend right, from=3-1, to=3-3]
	\arrow["{\fstProj'}", from=3-2, to=3-3]
  \end{tikzcd}
  \]
  we can conclude
  \begin{align*}
    & \snd(t \gg v_X) \\
    = &\, ((t \gg v_X) \times_B f) \gg \sndProj \\
    = &\, (t \times_B f) \gg (v_X \times_B f) \gg \sndProj \\
    = &\, (t \times_B f) \gg (\fstProj' \times_B \rho) \gg \sndProj' \\
    = &\, (\fst(t) \times_B \rho) \gg (t \times_B f') \gg \sndProj' \\
    = &\, (\fst(t) \times_B \rho) \gg \snd(t)
  \end{align*}

  In order to prove \Cref{eq:sndProj},
  let us also consider the following
  \begin{equation}\label{eq:vertical-nat-trans-computation-lemma-1}
  \begin{tikzcd}
    {\fstProj' \times_B f} & {\fstProj \times_B f} & {E \times X} \\
    {\fstProj' \times_B f'} && {E' \times X}
    \arrow["{v_X \times_B f}", from=1-1, to=1-2]
    \arrow["{\fstProj' \times_B \rho}"', from=1-1, to=2-1]
    \arrow["{\epsilon_X}", from=1-2, to=1-3]
    \arrow["{\rho \times X}", from=1-3, to=2-3]
    \arrow["{\epsilon'_X}"', from=2-1, to=2-3]
  \end{tikzcd}
  \end{equation}
  Recalling the formula $\sndProj = \epsilon_X \gg E^* X$
  given in the proof of \Cref{prop:poly-up},
  and assuming \Cref{eq:vertical-nat-trans-computation-lemma-1},
  we can calculate
  \begin{align*}
    & (v_X \times_B f) \gg \sndProj \\
    = \, & (v_X \times_B f) \gg \epsilon_X \gg E^* X \\
    = \, & (v_X \times_B f) \gg \epsilon_X \gg (\rho \times X) \gg E'^* X \\
    = \, & (\fstProj' \times_B \rho) \gg \epsilon'_X \gg E'^* X \\
    = \, & (\fstProj' \times_B \rho) \gg \sndProj'
  \end{align*}
  It remains to prove \Cref{eq:vertical-nat-trans-computation-lemma-1}.
  Starting with \Cref{eq:vertical-nat-trans-computation-lemma-0}
  in $\CC / B$,
  let us we reverse the direction of the isomorphism $f_* m_X$,
  and recognize ${\unit_{\fstProj'}} \gg {f_* n_{\fstProj'}}$
  as the transpose of
  $n_{\fstProj'} : f^* \fstProj' \to \rho^* f'^* \fstProj'$ in $\CC / E$
  under the partial adjunction $f^* \dashv f_*$
  \[
  \begin{tikzcd}[column sep = large]
    & {\fstProj'} & \fstProj & \fstProj \\
    {f_* f^*\fstProj'} & {f_*\rho^* f'^* \fstProj'} && {f_* \rho^* E'^* X}
    \arrow["{{{\nu_X}}}", from=1-2, to=1-3]
    \arrow["{{{\unit_{\fstProj'}}}}"', from=1-2, to=2-1]
    \arrow["{{\tilde{n}_{\fstProj'}}}", from=1-2, to=2-2]
    \arrow[equals, from=1-3, to=1-4]
    \arrow["{{{f_* m_X}}}", from=1-4, to=2-4]
    \arrow["{{{f_* n_{\fstProj'}}}}"', from=2-1, to=2-2]
    \arrow["{{{f_* \rho^* \counit_{E'^* X}}}}"', from=2-2, to=2-4]
  \end{tikzcd}\]
  Now taking the transpose of the diagram under the partial adjunction $f^* \dashv f_*$
  we obtain the following diagram in $\CC / E$
  \[
  \begin{tikzcd}[column sep = large]
    {f^* \fstProj'} & {f^*\fstProj} & {E^* X} \\
    {\rho^* f'^* \fstProj'} && {\rho^* E'^* X}
    \arrow["{f^* \nu_X}", from=1-1, to=1-2]
    \arrow["{{n}_{\fstProj'}}"', from=1-1, to=2-1]
    \arrow["{\counit_{E^*X}}", from=1-2, to=1-3]
    \arrow["{m_X}", from=1-3, to=2-3]
    \arrow["{\rho^* \counit_{E'^* X}}"', from=2-1, to=2-3]
  \end{tikzcd}
  \]
  Next, taking the transpose of the diagram under
  the partial adjunction $\rho_! \vdash \rho^*$
  (where $\rho_! : \CC / E \to \CC / E'$ and $\rho^* : \RR(E') \to \RR(E)$),
  we obtain the following diagram in $\CC / E'$
  \[
  \begin{tikzcd}[column sep = large]
    {\rho_! f^* \fstProj'} & {\rho_! f^*\fstProj} & {\rho_!E^* X} \\
    {f'^* \fstProj'} && {E'^* X}
    \arrow["{{\rho_! f^* \nu_X}}", from=1-1, to=1-2]
    \arrow["{\hat{n}_{\fstProj'}}"', from=1-1, to=2-1]
    \arrow["{{\rho_! \counit_{E^*X}}}", from=1-2, to=1-3]
    \arrow["{\hat{m}_X}", from=1-3, to=2-3]
    \arrow["{{\rho^* \counit_{E'^* X}}}"', from=2-1, to=2-3]
  \end{tikzcd}
  \]
  When we treat this as a diagram in $\CC$,
  this is precisely the required equation \Cref{eq:vertical-nat-trans-computation-lemma-1}.
\end{proof}

The natural isomorphism defined in the following lemma
is called the ``distributivity law''
in \cite{gambino2013}.
\begin{lemma} \label{lem:pushforward-along-pullback-morphism-action}
  Let $g : Z \to Y$ and $f : Y \to X$ be $\RR$-maps in a $\pi$-clan $(\CC,\RR)$.
  Consider the diagram
  \[\begin{tikzcd}
	Z & {f \times_X q} & {f_* Z} \\
	& Y & X
	\arrow["g"', from=1-1, to=2-2]
	\arrow["{{\epsilon}}"', from=1-2, to=1-1]
	\arrow["{{f'}}", from=1-2, to=1-3]
	\arrow[from=1-2, to=2-2]
	\arrow["\lrcorner"{anchor=center, pos=0.125}, draw=none, from=1-2, to=2-3]
	\arrow["{{q \, := \, f_* g}}", from=1-3, to=2-3]
	\arrow["f"', from=2-2, to=2-3]
  \end{tikzcd}\]
  where $\epsilon : f \times_X q \to Z$
  is the counit of the adjunction $f^* : \RR(X) \dashv \RR(Y) : f_*$.
  Then there is a canonical natural isomorphism
  $g_! \gg f_* \iso \epsilon^* \gg f'_* \gg q_!$
  \begin{equation} \label{eq:pushforward-along-pullback-morphism-action-2}
    \begin{tikzcd}
      {\RR(Z)} & {\RR (f \times_X q)} & {\RR (f_* Z) } \\
      & {\RR (Y)} & {\RR (X)}
      \arrow["{{\epsilon^*}}", from=1-1, to=1-2]
      \arrow["{{g_!}}"', from=1-1, to=2-2]
      \arrow["{{{f'_*}}}", from=1-2, to=1-3]
      \arrow["{{q_!}}", from=1-3, to=2-3]
      \arrow["{{{f}_*}}"', from=2-2, to=2-3]
    \end{tikzcd}
  \end{equation}
\end{lemma}
\begin{proof}
  We refer to \cite[Proposition 2.4.15.]{joyal2017}.
\end{proof}

We can use this isomorphism
  to analyze the action of $f_*$
  on an $\RR$-map $h : W \to Z$ between objects $W$ and $Z$ in $\RR(Y)$.
  Let us write $H$ for the object in the local clan $\RR(Z)$
  with underlying morphism $h$,
  $1_Z$ for the terminal object of $\RR(Z)$
  with underlying morphism $\id_Z$,
  and $\underline{H} : H \to 1_Z$ for the map such that $g_! \underline H = h$.
  Then 
  \begin{equation} \label{eq:pushforward-along-pullback-morphism-action-1}
  \begin{tikzcd}
    {f_*W} & {f_* (g_! H)} & {q_! (f'_* (\epsilon^* H))} \\
    {f_* Z} & {f_*(g_! 1_Z)} & {q_! (f'_* (\epsilon^* 1_Z))} & {q_! 1_{f_* Z}}
    \arrow[equals, from=1-1, to=1-2]
    \arrow["{{f_* h}}"', from=1-1, to=2-1]
    \arrow["\iso"{marking, allow upside down}, draw=none, from=1-2, to=1-3]
    \arrow["{{f_*(g_! \underline H)}}"', from=1-2, to=2-2]
    \arrow["{{q_! (f'_* (\epsilon^* \underline H))}}"', from=1-3, to=2-3]
    \arrow["{q_!(\underline {f'_*(\epsilon^* H)})}"{description}, from=1-3, to=2-4]
    \arrow[equals, from=2-1, to=2-2]
    \arrow["\iso"{description}, draw=none, from=2-2, to=2-3]
    \arrow["\iso"{description}, draw=none, from=2-3, to=2-4]
  \end{tikzcd}
\end{equation}

Thus, we can record the following result.

\begin{corollary} \label{cor:pushforward-stability}
  Under the assumptions and notation of
  \Cref{lem:pushforward-along-pullback-morphism-action},
  suppose $h : W \to Z$ is an $\RR$-map
  between objects $W$ and $Z$ in $\RR(Y)$.
  For a class of maps $\HH$ that is stable under
  pre- and post-composition with isomorphisms,
  we have $f_* h$ is an $\HH$-map if and only if  
  the object $f'_* (\epsilon^* h)$ in $\RR(f_* Z)$
  has an underlying map in $\HH$.
\end{corollary}

\section{Algebraic semantics}\label{sec:StrSem}


Fix a $\pi$-clan $(\CC, \RR)$.
$\RR$ is a class of distinguished families,
where each family (up to isomorphism)
can be thought of as a series of semantic context extensions
$\Gamma . A_1 . \cdots . A_n \to \Gamma$.
This view is justified later in
\Cref{thm:elementary-universe-to-algebraic-universe} and
\Cref{cor:elementary-universe-hierarchy-to-algebraic}.

\begin{definition}[\done{https://github.com/sinhp/HoTTLean/blob/TYPES2026/HoTTLean/Model/Structured/StructuredUniverse.lean\#L20}]
  \label{def:algebraic-universe}
  An \emph{algebraic universe} in $(\CC,\RR)$ is a universe in
  $\CC$ which is also an $\RR$-map.
  We will just say a universe $t$ is $\RR$-algebraic, or algebraic when the $\pi$-clan
  is clear from the context.
\end{definition}

We replicate the type formers described in \cite{awodey2025} in this setting,
making use of our general theory of polynomial functors.

\begin{definition}[$\Unit$-types]
  Suppose $\CC$ has a terminal object $1$.
  An algebraic $\Unit$-type structure on an algebraic universe
  $\tp:\Tm\to \Ty$ consists of maps
  $\Unit : 1 \to \Tm$ and $\unit : 1 \to \Ty$, such that
  \[\begin{tikzcd}
    1 & \Tm \\
    1 & \Ty
    \arrow["\unit", from=1-1, to=1-2]
    \arrow[equals, from=1-1, to=2-1]
    \arrow["\tp", from=1-2, to=2-2]
    \arrow["\Unit"', from=2-1, to=2-2]
  \end{tikzcd}\]
\end{definition}

\begin{definition}[$\Pi$-types, \done{https://github.com/sinhp/HoTTLean/blob/TYPES2026/HoTTLean/Model/Structured/StructuredUniverse.lean\#L359}]
    An algebraic $\Pi$-type structure on an
    algebraic universe $\tp:\Tm\to \Ty$ consists of maps
    $\Pi : P_\tp \Ty \to \Ty$ and
    $\lam:P_\tp \Tm\to \Tm$, such that
    \[ \begin{tikzcd}
      {P_{\tp}{\Tm}} & {\Tm} \\
      {P_{\tp}{\Ty}} & {\Ty}
      \arrow["\lam", from=1-1, to=1-2]
      \arrow["{P_{\tp}{\tp}}"', from=1-1, to=2-1]
      \arrow["\tp", from=1-2, to=2-2]
      \arrow["\Pi"', from=2-1, to=2-2]
    \end{tikzcd} \]
    is a pullback square.
\end{definition}


\begin{definition}[$\Sigma$-types, \done{https://github.com/sinhp/HoTTLean/blob/TYPES2026/HoTTLean/Model/Structured/StructuredUniverse.lean\#L651}]
    An algebraic $\Sigma$-type structure on an
    algebraic universe $\tp:\Tm\to \Ty$ consists of maps
    $\Sigma : P_\tp \Ty \to \Ty$ and $\pair: \compDom \to \Tm$,
    such that
    \[
    \begin{tikzcd}
        \compDom \ar[d,"\tp \pcomp \tp",swap]\ar[r,"\pair"] & \Tm \ar[d,"\tp"]\\
        P_\tp\Ty \ar[r,swap,"\Sigma"]& \Ty
    \end{tikzcd}
    \]
    is a pullback square.
    Here, we are using polynomial composition $\tp \pcomp \tp$
    from \Cref{def:pcomp}.
\end{definition}

\begin{definition}[Weak pullback, \done{https://github.com/sinhp/HoTTLean/blob/TYPES2026/HoTTLean/ForMathlib/CategoryTheory/WeakPullback.lean\#L10}]
  Consider the following square
  \[
  \begin{tikzcd}
    P & X \\
    Y & Z
    \arrow["{g'}", from=1-1, to=1-2]
    \arrow["{f'}"', from=1-1, to=2-1]
    \arrow["f", from=1-2, to=2-2]
    \arrow["g"', from=2-1, to=2-2]
  \end{tikzcd}
  \]
  A \emph{weak pullback structure} $l_\bullet$
  for this square consists of,
  for each cone $(W,x,y)$,
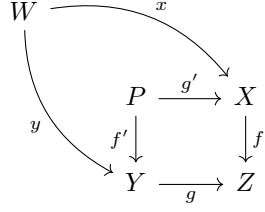
  \[
  \begin{tikzcd}
    W \\
    & P & X \\
    & Y & Z
    \arrow["x", bend left, from=1-1, to=2-3]
    \arrow["y"', bend right, from=1-1, to=3-2]
    \arrow["{g'}", from=2-2, to=2-3]
    \arrow["{f'}"', from=2-2, to=3-2]
    \arrow["f", from=2-3, to=3-3]
    \arrow["g"', from=3-2, to=3-3]
  \end{tikzcd}
  \]
  a chosen lift $l_W : W \to P$
  satisfying $l_W \gg g' = x$ and $l_W \gg f' = y$.
  We say that the weak pullback structure is
  \emph{coherent} when for any
  cone morphism $\sigma : V \to W$,
  the chosen lifts commute $\sigma \gg l_W = l_V$. 
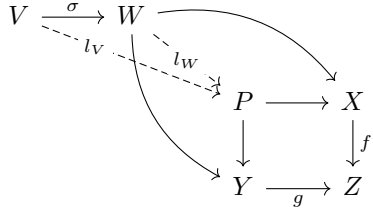
  \[
  \begin{tikzcd}
    V & W \\
    && P & X \\
    && Y & Z
    \arrow["\sigma", from=1-1, to=1-2]
    \arrow["{l_V}"{description, pos=0.3}, dashed, from=1-1, to=2-3]
    \arrow["{l_W}"{description}, dashed, from=1-2, to=2-3]
    \arrow[bend left, from=1-2, to=2-4]
    \arrow[bend right, from=1-2, to=3-3]
    \arrow[from=2-3, to=2-4]
    \arrow[from=2-3, to=3-3]
    \arrow["f", from=2-4, to=3-4]
    \arrow["g"', from=3-3, to=3-4]
  \end{tikzcd}
  \]
\end{definition}

\begin{lemma}[\done{https://github.com/sinhp/HoTTLean/blob/TYPES2026/HoTTLean/ForMathlib/CategoryTheory/WeakPullback.lean\#L31}] \label{lem:coherent-weak-pullback}
  When a (strong) pullback exists,
  any weak pullback structure can be replaced with a
  coherent weak pullback structure.
\end{lemma}
\begin{proof}
  Consider the following diagram,
  where $S$ is a (strong) pullback
  and $W$ has the structure of a weak pullback.
  
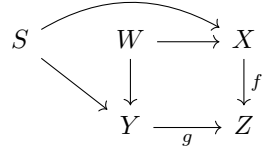
\[
  \begin{tikzcd}
    S & W & X \\
    & Y & Z
    \arrow[bend left, from=1-1, to=1-3]
    \arrow[from=1-1, to=2-2]
    \arrow[from=1-2, to=1-3]
    \arrow[from=1-2, to=2-2]
    \arrow["f", from=1-3, to=2-3]
    \arrow["g"', from=2-2, to=2-3]
  \end{tikzcd}  \]
  We define a new weak pullback structure $l'_\bullet$ for $W$.
  From the universal property of (strong) pullbacks,
  $S$ has a unique weak pullback structure,
  denoted $s_\bullet$,
  which is coherent.
  For any cone $V$,
  take $l'_V := \, s_V \gg l_S : V \to W$
  by composing the lift for $S$
  with the lift for $W$.
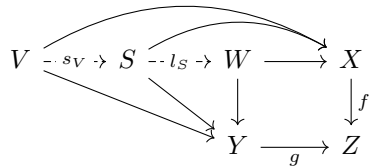
  \[
\begin{tikzcd}
	V & S & W & X \\
	&& Y & Z
	\arrow["{s_V}"{description}, dashed, from=1-1, to=1-2]
	\arrow[bend left, from=1-1, to=1-4]
	\arrow[from=1-1, to=2-3]
	\arrow["{l_S}"{description}, dashed, from=1-2, to=1-3]
	\arrow[bend left, from=1-2, to=1-4]
	\arrow[from=1-2, to=2-3]
	\arrow[from=1-3, to=1-4]
	\arrow[from=1-3, to=2-3]
	\arrow["f", from=1-4, to=2-4]
	\arrow["g"', from=2-3, to=2-4]
\end{tikzcd}\]
  Then $l'_\bullet$ is coherent since $s_\bullet$ is.
\end{proof}

\begin{definition}[$\Id$-types] \label{def:algebraic-id-types}
  An algebraic $\Id$-type structure on an algebraic universe
  $\tp : \Tm \to \Ty$ consists of the following
  \begin{enumerate}
  \item
    Consider the following pullback provided by context extension.
    
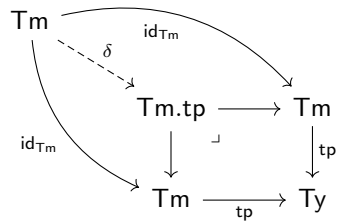
\[\begin{tikzcd}
	\Tm \\
	& {\Tm . \tp} & \Tm \\
	& \Tm & \Ty
	\arrow[dashed, "\delta", from=1-1, to=2-2]
	\arrow["{\id_\Tm}"', bend left, from=1-1, to=2-3]
	\arrow["{\id_\Tm}"', bend right, from=1-1, to=3-2]
	\arrow[from=2-2, to=2-3]
	\arrow[from=2-2, to=3-2]
	\arrow["\lrcorner"{anchor=center, pos=0.125}, draw=none, from=2-2, to=3-3]
	\arrow["\tp", from=2-3, to=3-3]
	\arrow["\tp"', from=3-2, to=3-3]
    \end{tikzcd}\]
    Here $\Tm.\tp$ is the context $X : \Ty.\,x : X.\,y : X$.
    Formation of identity types and reflexivity are
    modeled by a commutative square 
    \[
    \begin{tikzcd}
      \Tm\ar[d,"\delta", swap]\ar[r,"\refl"] & \Tm\ar[d,"\tp"] \\
      \Tm . \tp \ar[r,"\Id"] & \Ty
    \end{tikzcd}
    \]
  \item
    Then consider the context extension for $\Id$,
    and substitution $\rho := \id_\Tm . \id_\Tm . $
    \[\begin{tikzcd}
	\Tm \\
	& {\Tm . \tp . \Id} & \Tm \\
	& {\Tm . \tp} & \Ty
	\arrow["\rho"{description}, dashed, from=1-1, to=2-2]
	\arrow["\refl", bend left, from=1-1, to=2-3]
	\arrow["\delta"', bend right, from=1-1, to=3-2]
	\arrow[from=2-2, to=2-3]
	\arrow["{d_\Id}", from=2-2, to=3-2]
	\arrow["\lrcorner"{anchor=center, pos=0.125}, draw=none, from=2-2, to=3-3]
	\arrow["\tp", from=2-3, to=3-3]
	\arrow["\Id", from=3-2, to=3-3]
    \end{tikzcd}\]
    This creates a triangle $\rho \gg i = \id_\Tm$,
    \[
    \begin{tikzcd}
    \Tm\ar[dr,"\delta"]\ar[ddr,"\id_\Tm",swap] \ar[r,"\rho"]& \Tm . \tp . \Id \ar[d,"d_\Id"] \ar[dd,bend left = 60,"i"]\\
    &\Tm . \tp \ar[d,"d_\tp"] \\
     & \Tm
    \end{tikzcd}
    \]
    which induces a natural transformation
    $v : P_i\to P_{\id_\Tm} \iso \Tm \times (-)$
    by \Cref{prop:vertical-nat-trans-computation}.
    The $j$-rule,
    or $\Id$-elimination,
    is captured by requiring the following naturality square to have a weak pullback structure.
    \[\begin{tikzcd}
	{P_i\Tm} & {P_{\id_\Tm}\Tm} & {\Tm \times \Tm} \\
	{P_i\Ty} & {P_{\id_\Tm}\Ty} & {\Tm \times \Ty}
	\arrow["{v_\Tm}", from=1-1, to=1-2]
	\arrow["{P_i\tp}"', from=1-1, to=2-1]
	\arrow["\iso"{description}, draw=none, from=1-2, to=1-3]
	\arrow["{P_{\id_\Tm}\tp}", from=1-2, to=2-2]
	\arrow["{\Tm \times \tp}", from=1-3, to=2-3]
	\arrow["{v_\Ty}"', from=2-1, to=2-2]
	\arrow["\iso"{description}, draw=none, from=2-2, to=2-3]
    \end{tikzcd}\]
  \item The weak pullback structure is coherent.
  \end{enumerate} 
\end{definition}

The final condition is somewhat auxiliary
in the sense that any weak pullback structure
can be corrected into a coherent weak pullback structure
whenever a (strong) pullback exists, by \Cref{lem:coherent-weak-pullback}.
Indeed, $\Tm \times \tp$ is a pullback of $\tp$
which has all chosen (strong) pullbacks.

Unlike in \cite{awodey2025},
our identity types are defined in a Paulin-Mohring style,
like in \cite{nawrocki2026}
for the following two reasons.
Firstly (and rather trivially),
$i$ is a smaller expression than $i \gg \tp$,
resulting in smaller context extension calculations.
Secondly, it is convenient that we are taking
the polynomial functor with polynomial signature $\id_\Tm$,
which produces a product functor $\Tm \times (-)$.

\section{Translation between elementary and algebraic semantics}\label{sec:Translation}
The main contribution of the work presented in our paper is the definition of an algebraic model in a $\pi$-clan,
showing that algebraic semantics of MLTT need not assume LCC structure. 
The translation given in this section ensures the correctness of our definition.
The proof stems from the observation that any interpretation of types
and terms (from an elementary model, natural model, or CwF)
automatically produces a pullback stable class of maps $R$.
Specifically, for a universe $\tp : \Tm \to \Ty$,
one can take $R$ to be the class of pullbacks of $\tp$.
That is, all type families classified by the universe.
We have the following correspondence between
conditions on $\tp$ and conditions on $R$.

\begin{equation}\label{eq:tp-vs-R}
\begin{tabular}{c|c}
   property of $\tp$ & property of $R$\\
   \hline
   substitution  & stable under pullback \\
   $\Unit$-type  & contains all isomorphisms \\
   $\Sigma$-type & closed under composition \\
   $\Pi$-type & closed under pushforward \\
\end{tabular}
\end{equation}

Therefore, a $\pi$-preclan\footnote{
  See \Cref{def:clan}.} 
$(\CC,\RR)$ naturally arises from a universe $\tp$ with elementary
$\Unit$-, $\Sigma$-, and $\Pi$-types.
Furthermore, the universe $\tp$ is an $\RR$-map.

\subsection{Constructing an algebraic model from an elementary one}\label{sec:UnstrToStr}

For the sake of generality, we will phrase \Cref{prop:elementary-unit-to-algebraic-unit},
\Cref{prop:elementary-sigma-to-algebraic-sigma},
\Cref{prop:elementary-pi-to-algebraic-pi},
and \Cref{prop:elementary-id-to-algebraic-id}
with respect to an algebraic universe $\tp$,
i.e., with respect to a fixed $\pi$-clan.
The more delicate conditions for converting a universe into an algebraic one are stated in \Cref{cor:elementary-universe-hierarchy-to-algebraic}.

\begin{proposition} \label{prop:elementary-unit-to-algebraic-unit}
  If $\tp$ is an algebraic universe,
  elementary $\Unit$-types on $\tp$
  give rise to algebraic $\Unit$-types on $\tp$.
\end{proposition}
\begin{proof}
  By taking the elementary $\Unit$-type
  on the context $1$,
  we construct a commutative square
  \[\begin{tikzcd}
    1 & \Tm \\
    1 & \Ty
    \arrow["\unit_1", from=1-1, to=1-2]
    \arrow[equals, from=1-1, to=2-1]
    \arrow["\tp", from=1-2, to=2-2]
    \arrow["\Unit_1"', from=2-1, to=2-2]
  \end{tikzcd}\]
  It is a pullback since a cone of the diagram
  must be of the form
  \[
  \begin{tikzcd}
    \Gamma & 1 & \Tm \\
    & 1 & \Ty
    \arrow["{!}", dashed, from=1-1, to=1-2]
    \arrow["{\unit_\Gamma}", bend left, from=1-1, to=1-3]
    \arrow["{!}"', from=1-1, to=2-2]
    \arrow["{\unit_1}", from=1-2, to=1-3]
    \arrow[equals, from=1-2, to=2-2]
    \arrow["\tp", from=1-3, to=2-3]
    \arrow["{\Unit_1}"', from=2-2, to=2-3]
  \end{tikzcd}
  \]
\end{proof}

\begin{proposition}[\done{https://github.com/sinhp/HoTTLean/blob/TYPES2026/HoTTLean/Model/Structured/StructuredUniverse.lean\#L635}] \label{prop:elementary-pi-to-algebraic-pi}
  If $\tp$ is an algebraic universe,
  elementary $\Pi$-types on $\tp$
  give rise to algebraic $\Pi$-types on $\tp$.
\end{proposition} 
\begin{proof}
  To define a map $\Pi' : P_\tp \Ty \to \Ty$,
  it suffices to use the Yoneda embedding
  and define a natural transformation
  $y(\Pi') : y(P_\tp \Ty) \to y(\Ty)$.
  Suppose we have a map
  $\alpha : \Gamma \to P_\tp \Ty$.
  By \Cref{prop:poly-up}, this can be viewed as a pair of maps
  $A : \Gamma \to \Ty$ and $B : \Gamma . A \to \Ty$.
  Let us denote $\alpha$ as $(A,B)$.
  Using the elementary $\Pi$-type on $\tp$,
  we define the component of $y(\Pi')$ on $\Gamma$ by
  \begin{align*}
    y(\Pi')_\Gamma : \CC(\Gamma,P_\tp \Ty) & \to \CC(\Gamma, \Ty) \\
    (A,B) & \mapsto \Pi_A B
  \end{align*}
  Naturality of $y(\Pi')$ then follows from stability of
  $\Pi_A B$ under substitution.
  
  We can define $\lam' : P_\tp \Tm \to \Tm$
  in exactly the same way.
  To show that $\lam' \gg \tp = P_\tp \tp \gg \Pi'$
  it suffices to show that the same holds after applying the Yoneda embedding
  $y(\lam') \gg y(\tp) = y(P_\tp \tp) \gg y(\Pi')$.
  Indeed, let $(A,b) : \Gamma \to \Tm$,
  which we have decomposed as a pair
  $A : \Gamma \to \Ty$ and $b : \Gamma . A \to \Tm$,
  and denote the latter's type as $B := b \gg \tp$.

  \begin{align*}
    & y(\lam' \gg \tp)_\Gamma (A,b) \\
    = \, & y(\tp)_\Gamma (y(\lam')_\Gamma (A,b)) \\
    = \, & y(\tp)_\Gamma (\lam \, b) \\
    = \, & \lam \, b \gg \tp \\
    = \, & \Pi_A B \\
    = \, & y(\Pi')_\Gamma (A,B) \\
    = \, & y(\Pi')_\Gamma (y(P_\tp \tp)_\Gamma (A,b)) \\
    = \, & y(P_\tp \tp \gg \Pi')_\Gamma (A,b) \\
  \end{align*}

  To show that the square is a pullback,
  let $(A,B) : \Gamma \to P_\tp \Ty$ and $f : \Gamma \to \Tm$
  form a cone of the diagram,
  then $(A,\unlam f) : \Gamma \to P_\tp \Tm$ is the
  unique lift into the pullback.
  \[
  \begin{tikzcd}
    \Gamma \\
    & {P_{\tp}{\Tm}} & \Tm \\
    & {P_{\tp}{\Ty}} & \Ty
    \arrow["{(A,\unlam f)}"{description}, dashed, from=1-1, to=2-2]
    \arrow["f", bend left, from=1-1, to=2-3]
    \arrow["{(A,B)}"', bend right, from=1-1, to=3-2]
    \arrow["{\lam'}"', from=2-2, to=2-3]
    \arrow["{{P_{\tp}{\tp}}}"', from=2-2, to=3-2]
    \arrow["\tp", from=2-3, to=3-3]
    \arrow["{\Pi'}"', from=3-2, to=3-3]
  \end{tikzcd}
  \]
  This is because by the $\beta$- and $\eta$-rules,
  $\unlam \, f$ is the unique map
  $b : \Gamma . A \to \Tm$ such that
  $\lam \, b = f$.
\end{proof}

\begin{proposition}[\done{https://github.com/sinhp/HoTTLean/blob/TYPES2026/HoTTLean/Model/Structured/StructuredUniverse.lean\#L924}] \label{prop:elementary-sigma-to-algebraic-sigma}
  If $\tp$ is an algebraic universe,
  elementary $\Sigma$-types on $\tp$
  give rise to algebraic $\Sigma$-types on $\tp$.
\end{proposition}
\begin{proof}
  We can define $\Sigma' : P_\tp \Ty \to \Ty$
  in the same way as $\Pi'$ was defined in
  \Cref{prop:elementary-pi-to-algebraic-pi}.
  
  To define a map $\pair' : \compDom \to \Tm$,
  it suffices to use the Yoneda embedding
  and define a natural transformation
  $y(\pair') : y(\compDom) \to y(\Tm)$.
  Suppose we have a map
  $\alpha : \Gamma \to \compDom$.
  By \Cref{prop:poly-up},
  we can decompose $\alpha$ as the following maps
  \[
  \begin{tikzcd}
	& \Gamma \\
	\Tm & \compDom \\
	\Ty & {\fstProj \times_\Ty \tp} & \Tm \\
	& {P_\tp \Ty} & \Ty
	\arrow["b"', dashed, from=1-2, to=2-1]
	\arrow["\alpha", from=1-2, to=2-2]
	\arrow["a"{description}, bend left, dashed, from=1-2, to=3-3]
	\arrow["{(A,B)}"{description}, bend left = 60, dashed, from=1-2, to=4-2]
	\arrow["\tp"', from=2-1, to=3-1]
	\arrow[from=2-2, to=2-1]
	\arrow["\lrcorner"{anchor=center, pos=0.125, rotate=-90}, draw=none, from=2-2, to=3-1]
	\arrow[from=2-2, to=3-2]
	\arrow["{{\sndProj}}", from=3-2, to=3-1]
	\arrow[from=3-2, to=3-3]
	\arrow[from=3-2, to=4-2]
	\arrow["\lrcorner"{anchor=center, pos=0.125}, draw=none, from=3-2, to=4-3]
	\arrow["\tp", from=3-3, to=4-3]
	\arrow["{{\fstProj}}"', from=4-2, to=4-3]
\end{tikzcd}\]
  We have that $a \gg \tp = A$ and $b \gg \tp = \id_\Gamma . a \gg B$.
  Then we define $y(\pair')_\Gamma : \alpha \mapsto \pair (a,b)$.
  Naturality of $y(\pair')$ follows from stability of $\pair$ under substitution.
  \[
  \begin{tikzcd}
      \compDom \ar[d,"\tp \pcomp \tp",swap]\ar[r,"\pair'"] & \Tm \ar[d,"\tp"]\\
      P_\tp\Ty \ar[r,swap,"\Sigma'"]& \Ty
  \end{tikzcd}
  \]
  To show that the square commutes,
  we check that for $\alpha = (A,B,a,b) : \Gamma \to \compDom$
  \begin{align*}
    & y(\pair' \gg \tp)_\Gamma (A,B,a,b) \\
    = \, & y(\tp)_\Gamma (y(\pair')_\Gamma (A,B,a,b)) \\
    = \, & y(\tp)_\Gamma (\pair (a,b)) \\
    = \, & \pair (a,b) \gg \tp \\
    = \, & \Sigma_A B \\
    = \, & y(\Sigma')_\Gamma (A,B) \\
    = \, & y(\Sigma')_\Gamma (y(\tp \pcomp \tp)_\Gamma (A,B,a,b)) \\
    = \, & y(\tp \pcomp \tp \gg \Sigma')_\Gamma (A,B,a,b) \\
  \end{align*}
  To show that the square is a pullback,
  suppose $(A,B) : \Gamma \to P_\tp \Ty$ and $p : \Gamma \to \Tm$ form a cone,
  then $(A,B, \fst \, p, \snd \, p) : \Gamma \to P_\tp \Tm$ is the
  unique lift into the pullback.
  \[
  \begin{tikzcd}
    \Gamma \\
    & \compDom & \Tm \\
    & {P_\tp \Ty} & \Ty
    \arrow["{(A,B,\fst\, p,\snd\, p)}"{description}, dashed, from=1-1, to=2-2]
    \arrow["p", bend left, from=1-1, to=2-3]
    \arrow["{(A,B)}"', bend right = 60, from=1-1, to=3-2]
    \arrow["{\pair'}", from=2-2, to=2-3]
    \arrow["{\tp \pcomp \tp}"', from=2-2, to=3-2]
    \arrow["\tp", from=2-3, to=3-3]
    \arrow["{\Sigma'}"', from=3-2, to=3-3]
  \end{tikzcd}
  \]
  If $(A,B,a,b) : \Gamma \to \compDom$,
  then $(A,B,a,b) \gg \lam' = p$
  if and only if $\pair(a,b) = p$,
  if and only if $a = \fst \, p$ and $b = \snd \, p$,
  by the $\beta$ and $\eta$-rules.
\end{proof}

\begin{proposition} \label{prop:elementary-id-to-algebraic-id}
  If $\tp$ is an algebraic universe,
  elementary $\Id$-types on $\tp$
  give rise to algebraic $\Id$-types on $\tp$.
  The weak pullback structure constructed this way
  is automatically {coherent} when the elementary elimination $j$
  is stable under substitution.
\end{proposition}
\begin{proof}
\begin{enumerate}
  \item
    To define a map $\Id' : \Tm . \tp \to \Ty$,
    it suffices to use the Yoneda embedding
    and define a natural transformation
    $y(\Id') : y(\Tm . \tp) \to y(\Ty)$.
    Suppose we have a map
    $\alpha : \Gamma \to \Tm . \tp$.
    This can be viewed as a pair of maps
    $a_0, a_1 : \Gamma \to \Tm$
    such that $a_0 \gg \tp = a_1 \gg \tp$.
    Let us use $A := a_0 \gg \tp$ to denote this composition.
    Using the elementary $\Id$-type on $\tp$,
    we define $y(\Id')_\Gamma : \alpha \mapsto \Id_A (a_0,a_1)$.
    Naturality of $y(\Id')$ then follows from stability of
    $\Id_A(a_0,a_1)$ under substitution.
    
    Next, we define $\refl' : \Tm \to \Tm$ in the same manner,
    by making a natural transformation $y(\refl') : y(\Tm) \to y(\Tm)$.
    Suppose we have a map $a : \Gamma \to \Tm$.
    Using the unstrucutred $\Id$-type on $\tp$,
    we define $y(\refl')_\Gamma : a \mapsto \refl \, a$.
    Then naturality of $y(\refl')$ follows from stability of $\refl a$
    under substitution.
    
    To show that $\refl' \gg \tp = \delta \gg \Id'$,
    it suffices to show that this is
    true after applying the Yoneda embedding
    \[
    \begin{tikzcd}
      \Tm\ar[d,"\delta", swap]\ar[r,"\refl'"] & \Tm\ar[d,"\tp"] \\
      \Tm . \tp \ar[r,"\Id'"] & \Ty
    \end{tikzcd}
    \quad \rightsquigarrow
    \begin{tikzcd}
      y(\Tm)\ar[d,"y(\delta)", swap]\ar[r,"y(\refl')"] & y(\Tm)
      \ar[d,"y(\tp)"] \\
      y(\Tm . \tp) \ar[r,"y(\Id')"] & y(\Ty)
    \end{tikzcd}
    \]
    Indeed, suppose $a : \Gamma \to \Tm$.
    Denote $A := a \gg \tp$.
    Then 
    \[ 
      y(\tp)_\Gamma (y(\refl')_\Gamma (a))
      = y(\tp)_\Gamma (\refl \, a)
      = \refl \, a \gg \tp
      = \Id_A (a,a)
      = y(\Id')_\Gamma (a \gg \delta)
      = y(\Id')_\Gamma (y(\delta)_\Gamma (a))
    \]
  \item
    Suppose we have a cone
    \[
  \begin{tikzcd}
	\Gamma \\
	& {P_i\Tm} & {\Tm \times \Tm} \\
	& {P_i\Ty} & {\Tm \times \Ty}
	\arrow["{(a,c_\refl)}", from=1-1, to=2-3]
	\arrow["{(a,C)}"', from=1-1, to=3-2]
	\arrow["{{v_\Tm}}", from=2-2, to=2-3]
	\arrow["{{P_i\tp}}", from=2-2, to=3-2]
	\arrow["{{\Tm \times \tp}}", from=2-3, to=3-3]
	\arrow["{{v_\Ty}}"', from=3-2, to=3-3]
  \end{tikzcd}
  \]
  By the universal property of polynomial functors \Cref{prop:poly-up}
  and the computation in
  \Cref{prop:vertical-nat-trans-computation},
  we can decompose $(a,c_\refl)$ and $(a,C)$ as three maps
  $a : \Gamma \to \Tm$ and $c_\refl : \Gamma \to \Tm$
  and $C : a \times_\Tm i \to \Ty$ such that
  \[ c_\refl \gg \tp = (a \times_\Tm \rho) \gg C\]
  We want to construct a (coherent) lift $(a,c) : \Gamma \to P_i \Tm$,
  which, again by \Cref{prop:poly-up} and
  \Cref{prop:vertical-nat-trans-computation}
  we can do by combining two maps
  $a : \Gamma \to \Tm$ and $c : a \times_\Tm i \to \Tm$ satisfying
  \[ c \gg \tp = C
  \quad \text{ and } \quad
  a \times_\Tm \rho \gg c = c_\refl \]
  Let us rephrase the problem by
  replacing our notation for pullbacks with context extensions.
  The pullback $a \times_\Tm i$ is $\Gamma . (x : A) . \Id_A (a,x)$
  and $a \times_\Tm \rho$ is the substitution $\rho_a$
  in the following diagram
    \[\begin{tikzcd}
	\Gamma & \Tm \\
	{\Gamma . (x : A) . \Id_A(a,x)} & {\Tm . \tp . \Id} \\
	{\Gamma.(x : A)} & {\Tm . \tp} \\
	\Gamma & \Tm
	\arrow["a", from=1-1, to=1-2]
	\arrow["{\rho_a}"', from=1-1, to=2-1]
	\arrow["\lrcorner"{anchor=center, pos=0.125}, draw=none, from=1-1, to=2-2]
	\arrow["\rho", from=1-2, to=2-2]
	\arrow[from=2-1, to=2-2]
	\arrow["{d_{\Id_A(a,x)}}"', from=2-1, to=3-1]
	\arrow["\lrcorner"{anchor=center, pos=0.125}, draw=none, from=2-1, to=3-2]
	\arrow["{d_\Id}", from=2-2, to=3-2]
	\arrow[from=3-1, to=3-2]
	\arrow["{d_A}"', from=3-1, to=4-1]
	\arrow["\lrcorner"{anchor=center, pos=0.125}, draw=none, from=3-1, to=4-2]
	\arrow["{d_\tp}", from=3-2, to=4-2]
	\arrow["a"', from=4-1, to=4-2]
    \end{tikzcd}\]
    Thus $C : \Gamma .(x : A) . \Id_A(a,x) \to \Ty$
    is a motive and
    \[ c_\refl \gg \tp = \rho_a \gg C\]
    The elementary $\Id$-elimination gives us a term
    $j (C, c_\refl) : \Gamma . (x : A) . \Id_A(a,x) \to \Tm$
    satisfying $j \gg \tp = C$ and
    $\rho_a \gg j (C, c_\refl) = c_\refl$.
    Thus we simply conclude by taking $c := j(C, c_\refl)$.
    
    The lift is automatically coherent if $j$ is stable under
    substitution.
    As is the case in general,
    if it were not coherent we could replace it with a coherent
    lift using the (strong) pullback by \Cref{lem:coherent-weak-pullback}.
  \end{enumerate}
\end{proof}

So far,
we have only converted elementary type formers on
a universe to algebraic type formers.
The following theorem provides sufficient conditions
to convert an elementary \emph{universe} into an algebraic one. 




\begin{definition}
  \label{def:principal-class}
  Let $t : Y \to X$ be a morphism in $\CC$.
  We will call the class of all pullbacks of $f$
  \[ \RR_t := \{ \, f : E \to B \; | \; \text{ there exists \, }
  \begin{tikzcd}
    E & Y \\
    B & X
    \arrow[from=1-1, to=1-2]
    \arrow["f"', from=1-1, to=2-1]
    \arrow["\lrcorner"{anchor=center, pos=0.125}, draw=none, from=1-1, to=2-2]
    \arrow["t", from=1-2, to=2-2]
    \arrow[from=2-1, to=2-2]
  \end{tikzcd} \}\]
  the \emph{principal class} generated by $f$.
  We will switch out the word ``class'' for ``preclan''
  (or ``clan'', and so on) when appropriate.
\end{definition}

\begin{theorem}\label{thm:elementary-universe-to-algebraic-universe}
  Suppose $\tp : \Tm \to \Ty$ is a universe 
  with elementary $\Unit$-, $\Sigma$-, and $\Pi$-types.
  The principal class $\RR_\tp$ (\Cref{def:principal-class}) generated by $\tp$
  forms a $\pi$-preclan $(\CC,\RR_\tp)$.
  Since $\tp$ is a pullback of itself along the identity $\id : \Ty \to \Ty$,
  $\tp$ is in $\RR_\tp$.
\end{theorem}
\begin{proof}
  Every pullback $f : E \to B$ of $\tp$ is isomorphic to a context extension
  $d_A : \Gamma . A \to \Gamma$,
  where $\Gamma := B$ and $A : \Gamma \to \Ty$ is the given map in the pullback square
  $B \to \Ty$.
  Thus, without loss of generality we will assume such maps are context extensions
  whenever convenient.

  \begin{enumerate}
  \item
    The map $\tp$ has chosen pullbacks.
    Thus by pullback pasting, 
    pullbacks of all $\RR$-maps along all morphisms exist and are $\RR$-maps.
  \item
    Let $f : E \to B$ be an isomorphism.
    Then the elementary type former $\Unit_B : B \to \Ty$ classifies $f$.
    To show this,
    it suffices to show that $d_{\Unit_B} : B . \Unit_B \to B$ is an isomorphism,
    with the substitution $u := \id_B . \unit_B : B \to B . \Unit_B$
    as its inverse.
    \[
    \begin{tikzcd}
    B \\
    & {B . \Unit_B} & \Tm \\
    & B & \Ty
    \arrow["{u}"{description}, dashed, from=1-1, to=2-2]
    \arrow["{\unit_B}", bend left, from=1-1, to=2-3]
    \arrow["{\id_B}"', bend right, from=1-1, to=3-2]
    \arrow["{\var_{\Unit_B}}",from=2-2, to=2-3]
    \arrow["{d_{\Unit_B}}"{description}, from=2-2, to=3-2]
    \arrow["\lrcorner"{anchor=center, pos=0.125}, draw=none, from=2-2, to=3-3]
    \arrow["\tp", from=2-3, to=3-3]
    \arrow["{\Unit_B}"', from=3-2, to=3-3]
    \end{tikzcd}
    \]
    This immediately satisfies $u \gg d_{\Unit_B} = \id_B$.
    Conversely, by uniqueness of maps into the pullback,
    to show that $d_{\Unit_B} \gg u = \id_{B . \Unit_B}$
    it suffices to note that
    \[ d_{\Unit_B} \gg u \gg d_{\Unit_B} = d_{\Unit_B} \gg \id_B = \id_{B . \Unit_B} \gg d_{\Unit_B} \]
    and 
    \[ d_{\Unit_B} \gg u \gg \var_{\Unit_B} = \id_{B . \Unit_B} \gg \var_{\Unit_B} \]
    The second equation holding because both sides are the only term of type
    $\Unit_{B . \Unit_B}$.
    \[ d_{\Unit_B} \gg u \gg \var_{\Unit_B} \gg \tp = d_{\Unit_B} \gg \unit_B \gg \tp =
    d_{\Unit_B} \gg \Unit_{B . \Unit_B} = \Unit_{B . \Unit_B}\]
    and 
    \[ \id_{B . \Unit_B} \gg \var_{\Unit_B} \gg \tp = \var_{\Unit_B} \gg \tp = d_{\Unit_B} \gg \Unit_{B . \Unit_B} = \Unit_{B . \Unit_B}\]
    \item
      Suppose $d_B : \Gamma . A . B \to \Gamma . A$ and $d_A : \Gamma . A \to \Gamma$
      are two composable context extensions.
      We want to show that $d_B \gg d_A$ is also in $\RR$.
      It suffices to show that it is classified by the map $\Sigma_A B : \Gamma \to \Ty$,
      i.e. construct an isomorphism $p : \Gamma . A . B \to \Gamma . \Sigma_A B$
      such that $p \gg d_{\Sigma_A B} = d_B \gg d_A$.
      Consider the map $p$ given by
      \[
      \begin{tikzcd}
        {\Gamma . A . B} & {\Gamma . \Sigma_A B} & \Tm \\
        {\Gamma . A} & \Gamma & \Ty
        \arrow["p"', dashed, from=1-1, to=1-2]
        \arrow["{\pair(d_B \gg \var_A , \var_B)}", bend left, from=1-1, to=1-3]
        \arrow["{d_B}"', from=1-1, to=2-1]
        \arrow[from=1-2, to=1-3]
        \arrow["{d_{\Sigma_A B}}", from=1-2, to=2-2]
        \arrow["\lrcorner"{anchor=center, pos=0.125}, draw=none, from=1-2, to=2-3]
        \arrow["\tp", from=1-3, to=2-3]
        \arrow["{d_A}"', from=2-1, to=2-2]
        \arrow["{\Sigma_A B}"', from=2-2, to=2-3]
      \end{tikzcd}\]
      Using \Cref{lem:substituted-pair},
      to check the commutativity of the outer square,
      it suffices to check two simpler squares commute:
      \[
      \begin{tikzcd}[column sep = large]
        {\Gamma . A . B} & \Tm &&& {\Gamma . A . B} & \Tm \\
        \Gamma & \Ty &&& {\Gamma . A} & \Ty
        \arrow["{d_B \gg \var_A}", from=1-1, to=1-2]
        \arrow["{d_B \gg d_A}"', from=1-1, to=2-1]
        \arrow["\tp", from=1-2, to=2-2]
        \arrow["{\var_B}", from=1-5, to=1-6]
        \arrow["{{(d_B \gg d_A) . (d_B \gg \var_A) = d_B}}"', from=1-5, to=2-5]
        \arrow["\tp", from=1-6, to=2-6]
        \arrow["A"', from=2-1, to=2-2]
        \arrow["B"', from=2-5, to=2-6]
      \end{tikzcd}
      \]
      Next, our candidate for the inverse of $p$ is
      given by the following.
      By \Cref{lem:substituted-pair},
      the square
      \[\begin{tikzcd}
        {\Gamma . \Sigma_A B} & \Tm \\
        \Gamma & \Ty
        \arrow["{\var_{\Sigma_A B}}", from=1-1, to=1-2]
        \arrow["{d_{\Sigma_A B}}"', from=1-1, to=2-1]
        \arrow["\tp", from=1-2, to=2-2]
        \arrow["{\Sigma_A B}"', from=2-1, to=2-2]
      \end{tikzcd}\]
      produces two squares
      \[
      \begin{tikzcd}[column sep = large]
        {\Gamma . \Sigma_A B} & \Tm && {\Gamma . \Sigma_A B} & {\Gamma . A . B} & \Tm \\
        \Gamma & \Ty &&& {\Gamma . A} & \Ty
        \arrow["{\fst(\var_{\Sigma_A B})}", from=1-1, to=1-2]
        \arrow["{d_{\Sigma_A B}}"', from=1-1, to=2-1]
        \arrow["\tp", from=1-2, to=2-2]
        \arrow["p^{-1}", dashed, from=1-4, to=1-5]
        \arrow["{\snd(\var_{\Sigma_A B})}", bend left, from=1-4, to=1-6]
        \arrow["{d_{\Sigma_A B} . \fst(\var_{\Sigma_A B})}"', from=1-4, to=2-5]
        \arrow[from=1-5, to=1-6]
        \arrow["{d_B}", from=1-5, to=2-5]
        \arrow["\tp", from=1-6, to=2-6]
        \arrow["A"', from=2-1, to=2-2]
        \arrow["B"', from=2-5, to=2-6]
      \end{tikzcd}
      \]
      By the universal property of pullbacks,
      to show that $p \gg p^{-1} = \id_{\Gamma . A . B}$ it suffices to check
      the following three equations
      \begin{align*}
        & p \gg p^{-1} \gg d_B \gg d_A && p \gg p^{-1} \gg d_B \gg \var_A && p \gg p^{-1} \gg \var_B \\
        = & \, p \gg (d . \fst(\var)) \gg d_A &
        = & \, p \gg (d . \fst(\var)) \gg var_A &
        = & \, p \gg \snd(\var) \\
        = & \, p \gg d &
        = & \, p \gg \fst(\var) &
        = & \, \snd(p \gg \var) \\
        = & \, d_B \gg d_A &
        = & \, \fst(p \gg \var) &
        = & \, \snd(\pair(d_B \gg \var_A, \var_B)) \\
         & \,  &
        = & \, \fst(\pair(d_B \gg \var_A, \var_B)) &
        = & \, \var_B \\
         & \,  &
        = & \, d_B \gg \var_A &
         & \, 
      \end{align*}
      Similarly,
      to show that $p^{-1} \gg p = \id_{\Gamma . \Sigma_A B}$ it suffices to check the
      following two equations
      \begin{align*}
        & p^{-1} \gg p \gg d_{\Sigma_A B} && p^{-1} \gg p \gg \var_{\Sigma_A B} \\
        = & \, p^{-1} \gg d_B \gg d_A &
        = & \, p^{-1} \gg \pair(d_B \gg \var_A, \var_B) \\
        = & \, (d_{\Sigma_A B} . \fst(\var)) \gg d_A &
        = & \, \pair(p^{-1} \gg d_B \gg \var_A, p^{-1} \gg \var_B) \\
        = & \, d_{\Sigma_A B} &
        = & \, \pair((d . \fst(\var_{\Sigma_A B})) \gg \var_A, \snd(\var_{\Sigma_A B})) \\
         & \,  &
        = & \, \pair(\fst(\var_{\Sigma_A B}), \snd(\var_{\Sigma_A B})) \\
         & \,  &
        = & \, \var_{\Sigma_A B}
      \end{align*}
    \item
      Consider again two composable context extensions
      $d_B : \Gamma . A . B \to \Gamma . A$ and $d_A : \Gamma . A \to \Gamma$.
      It will suffice to show that $d_{\Pi_A B} : \Gamma . \Pi_A B \to \Gamma$
      is a pushforward ${d_A}_* d_B$,
      i.e. for any $\sigma : \Delta \to \Gamma$ there is a natural bijection 
      \[ l : {\CC / \Gamma . A} \, (d_A^* \sigma, d_B) \iso {\CC / \Gamma} \, (\sigma, d_{\Pi_A B}) \]
      As before, let us denote $\tilde \sigma := d_A^* \sigma$.
      Suppose $b : \tilde \sigma \to d_B$ in $\CC / \Gamma . A$.
      \[
      \begin{tikzcd}[column sep = large]
        & {\Gamma . A . B} \\
        {\Delta . (\sigma \gg A)} & {\Gamma . A} && {\Delta . (\sigma \gg A)} & \Tm \\
        \Delta & \Gamma && {\Gamma . A} & \Ty
        \arrow["{d_B}", from=1-2, to=2-2]
        \arrow["b", from=2-1, to=1-2]
        \arrow["{\tilde \sigma}"', from=2-1, to=2-2]
        \arrow["{d_{\sigma \gg A}}"', from=2-1, to=3-1]
        \arrow["{d_A}", from=2-2, to=3-2]
        \arrow["{b \gg \var_B}", from=2-4, to=2-5]
        \arrow["{\tilde \sigma}"', from=2-4, to=3-4]
        \arrow["\tp", from=2-5, to=3-5]
        \arrow["\sigma"', from=3-1, to=3-2]
        \arrow["B"', from=3-4, to=3-5]
      \end{tikzcd}
      \]
      By \Cref{lem:substituted-lam},
      this results in a term of the substituted $\Pi$-type,
      which gives us our desired map
      $l(b) : \sigma \to d_{\Pi_A B}$ in ${\CC / \Gamma}$.
      \[
      \begin{tikzcd}
        \Delta & {\Gamma. \Pi_A B} & \Tm \\
        & \Gamma & \Ty
        \arrow["{l(b)}"{description}, dashed, from=1-1, to=1-2]
        \arrow["{\lam (b \gg \var_B)}", bend left, from=1-1, to=1-3]
        \arrow["\sigma"', from=1-1, to=2-2]
        \arrow[from=1-2, to=1-3]
        \arrow["{d_{\Pi_A B}}", from=1-2, to=2-2]
        \arrow["\tp", from=1-3, to=2-3]
        \arrow["{\Pi_A B}"', from=2-2, to=2-3]
      \end{tikzcd}
      \]
      Conversely,
      given $f : \sigma \to d_{\Pi_A B}$,
      we similarly obtain
      \[
      \begin{tikzcd}
        {\Delta . (\sigma \gg A)} & {\Gamma . A . B} & \Tm \\
        & {\Gamma . A} & \Ty
        \arrow["{l^{-1}(b)}", dashed, from=1-1, to=1-2]
        \arrow["{\unlam (f \gg \var_{\Pi_A B})}", bend left, from=1-1, to=1-3]
        \arrow["{\tilde \sigma}"', from=1-1, to=2-2]
        \arrow[from=1-2, to=1-3]
        \arrow[from=1-2, to=2-2]
        \arrow["\tp", from=1-3, to=2-3]
        \arrow["B"', from=2-2, to=2-3]
      \end{tikzcd}\]
      It follows from the universal property of pullbacks
      and \Cref{lem:substituted-lam} that $l$ and
      $l^{-1}$ are inverses of each other.

      Furthermore $l$ is natural in $\sigma \in \CC / \Gamma$:
      for any $\tau : \Xi \to \Delta$ we have
      $l(\tilde \tau \gg b) = \tau \gg l(b)$.
      We check this using the universal property of pullbacks
      \[
      l(\tilde \tau \gg b) \gg d_{\Pi_A B}
      = \tau \gg \sigma
      = \tau \gg l(b) \gg d_{\Pi_A B}
      \]
      and
      
      \[
      l(\tilde \tau \gg b) \gg \var_{\Pi_A B}
      = \lam (\tilde \tau \gg b \gg \var_B)
      = \tau \gg \lam (b \gg \var_B)
      = \tau \gg l(b) \gg \var_{\Pi_A B}
      \]

  \end{enumerate}
\end{proof}
Each type former corresponds independently
to an axiom for $\pi$-preclans (see \Cref{eq:tp-vs-R}),
and we could have stated this theorem more generally
to emphasize those correspondences.
The failure of this $\pi$-preclan to be a full $\pi$-clan
comes down to a \textsf{Type : Type} issue that we would like to circumvent,
which would arise if $\Ty \to 1$ were a pullback of $\tp$.
The following theorem resolves this by
requiring an infinite hierarchy of universes.

\begin{corollary}\label{cor:elementary-universe-hierarchy-to-algebraic}
  Suppose $\CC$ has a terminal object
  and we have a sequence of (non-algebraic) universe lifts
  indexed by the naturals
  \[
  \begin{tikzcd}
    {\tp_0} & {\tp_1} & {\tp_2} & \cdots
    \arrow["{l_0^1}", from=1-1, to=1-2]
    \arrow["{l_1^2}", from=1-2, to=1-3]
    \arrow[from=1-3, to=1-4]
  \end{tikzcd}
  \]
  Additionally,
  suppose each $\tp_n$ has
  elementary $\Unit$-, $\Sigma$-, and $\Pi$-types.
  We use $(\CC,\RR_n)$ to denote the principal $\pi$-preclan
  generated by $\tp_n$
  (using \Cref{thm:elementary-universe-to-algebraic-universe}).
  Then the union of them all
  \[ \RR := \bigcup_{n \in \N} \RR_n \]
  forms a $\pi$-clan $(\RR(1), \RR)$,
  when we restrict to the subcategory of $\RR$-objects.
  Each universe $\tp_n$
  is $(\RR(1), \RR)$-algebraic.
  Furthermore, each algebraic universe $\tp_n$ has
  {algebraic} $\Unit$-, $\Sigma$-, and $\Pi$-types.
\end{corollary}
\begin{proof}
  Since universe lifts $l_n^{n+1}$ provide pullback squares
  for lifting types and terms,
  \[\begin{tikzcd}
    {\Tm_n} & {\Tm_{n+1}} \\
    {\Ty_n} & {\Ty_{n+1}}
    \arrow[from=1-1, to=1-2]
    \arrow["{{\tp_n}}"', from=1-1, to=2-1]
    \arrow["\lrcorner"{anchor=center, pos=0.125}, draw=none, from=1-1, to=2-2]
    \arrow["{{\tp_{n+1}}}", from=1-2, to=2-2]
    \arrow[from=2-1, to=2-2]
  \end{tikzcd}\]
  we have that the preclans are nested
  \[
    \RR_0 \subseteq \RR_1 \subseteq \cdots \subseteq \RR
  \]

  In general, restricting to the subcategory of $\RR$-objects
  converts a preclan $(\CC, \RR)$ into a clan $(\RR(1),\RR)$.
  We check that $(\CC, \RR)$ is an $\pi$-preclan:
  \begin{enumerate}
    \item
      Pullbacks of $\RR$-maps exist and are in $\RR$ since
      this holds for each $\RR_n$.
    \item
      All isomorphisms are in $\RR_0 \subseteq \RR$.
    \item
      If $g : Z \to Y$ and $f : Y \to X$ are in $\RR$,
      then there will be two naturals such that $g \in \RR_n$
      and $f \in \RR_m$
      and so $g$ and $f$ are both in $\RR_{\max(m,n)}$.
      We conclude by noting $\RR_{\max(m,n)} \subseteq \RR$ is closed under composition.
    \item Analogous to (3).
  \end{enumerate}
  
  We conclude the proof by applying 
  \Cref{prop:elementary-unit-to-algebraic-unit},
  \Cref{prop:elementary-sigma-to-algebraic-sigma}
  and \Cref{prop:elementary-pi-to-algebraic-pi}.
  Note that the theory of algebraic $\Sigma$ and $\Pi$-types
  requires applying polynomial functors to $\tp_n : \Tm_n \to \Ty_n$,
  so one must check that $\Tm_n$ and $\Ty_n$ are $\RR$-objects
  in order to apply \Cref{prop:elementary-sigma-to-algebraic-sigma}
  and \Cref{prop:elementary-pi-to-algebraic-pi}.
  Indeed, the universe lift $l_n^{n+1}$ provides a classifier for $\Ty_n$.
  \[
    \begin{tikzcd}
    {\Ty_n} & {\Tm_{n+1}} \\
    1 & {\Ty_{n+1}}
    \arrow["{{\mathsf{asTm}}}", from=1-1, to=1-2]
    \arrow[from=1-1, to=2-1]
    \arrow["\lrcorner"{anchor=center, pos=0.125}, draw=none, from=1-1, to=2-2]
    \arrow["{\tp_{n+1}}", from=1-2, to=2-2]
    \arrow["{{{U_n}}}"', from=2-1, to=2-2]
  \end{tikzcd}
  \]
  Hence $\Ty_n$ is an $\RR_{n+1}$-object.
  Since $\tp_n$ is an $\RR_{n+1}$-map and $\RR_{n+1}$
  is closed under composition,
  we also have $\Tm_n$ is an $\RR_{n+1}$-object.
  One can similarly check these conditions for \Cref{prop:elementary-id-to-algebraic-id}.
\end{proof}

\subsection{Constructing an elementary model from an algebraic one}

Since the proofs of the following theorems are very similar to their converses
in the previous section, we only proving a sketch.

\begin{proposition} \label{prop:algebraic-unit-to-elementary-unit}
  If $\tp$ is an algebraic universe,
  algebraic $\Unit$-types on $\tp$
  give rise to elementary $\Unit$-types on $\tp$.
\end{proposition}
\begin{proof}
  We can define the elementary $\Unit$-type by
  precomposing the algebraic $\Unit$-type with
  the unique map to the terminal object.
\end{proof}

\begin{proposition}[\done{https://github.com/sinhp/HoTTLean/blob/TYPES2026/HoTTLean/Model/Structured/StructuredUniverse.lean\#L802}] \label{thm:algebraic-sigma-to-elementary-sigma}
  If $\tp$ is an algebraic universe,
  algebraic $\Sigma$-types on $\tp$
  give rise to elementary $\Sigma$-types on $\tp$.
\end{proposition}
\begin{proof}
  For a pair $A : \Gamma \to \Ty$ and $B : \Gamma . A \to \Ty$,
  we can define the elementary $\Sigma$-type by
  composing $(A,B) : \Gamma \to P_\tp \Ty$
  with the algebraic $\Sigma$-type $P_\tp \Ty \to \Ty$.
\end{proof}

\begin{proposition}[\done{https://github.com/sinhp/HoTTLean/blob/TYPES2026/HoTTLean/Model/Structured/StructuredUniverse.lean\#L556}] \label{thm:algebraic-pi-to-elementary-pi}
  If $\tp$ is an algebraic universe,
  algebraic $\Pi$-types on $\tp$
  give rise to elementary $\Pi$-types on $\tp$.
\end{proposition}
\begin{proof}
  For a pair $A : \Gamma \to \Ty$ and $B : \Gamma . A \to \Ty$,
  we can define the elementary $\Pi$-type by
  composing $(A,B) : \Gamma \to P_\tp \Ty$
  with the algebraic $\Pi$-type $P_\tp \Ty \to \Ty$.
\end{proof}

\begin{proposition} \label{prop:algebraic-id-to-elementary-id}
  If $\tp$ is an algebraic universe,
  algebraic $\Id$-types on $\tp$
  give rise to elementary $\Id$-types on $\tp$.
  If the weak pullback structure provided is coherent,
  then $j$ is stable under substitution.
\end{proposition}
\begin{proof}
Fix $A:\Gamma\to \Ty$ and $a: \Gamma\to \Tm$ such that
$a \gg \tp = A$.
\begin{enumerate}
    \item
      For any $b : \Gamma \to \Tm$ such that $b \gg A = \tp$,
      Define $\Id_A(a,b): \Gamma\to \Ty$ as the composition:
      \[
      \begin{tikzcd}
        \Gamma \ar[r,"a.b"]& \Tm . \tp \ar[r,"\Id"] & \Ty_1
      \end{tikzcd}
      \]
    \item
      Stability of $\Id_A(a,b)$ under substitution
      follows from uniqueness of maps into pullbacks
      \[ \Id_{\sigma \gg A}(\sigma \gg a, \sigma \gg b)
      = ((\sigma \gg a) . (\sigma \gg b)) \gg \Id
      = \sigma \gg (a . b) \gg \Id
      \]
    \item
      The reflexive path constructor $\refl_a$ is the composition
    \[
    \begin{tikzcd}
        \Gamma \ar[r,"{a}"]& \Tm\ar[r,"\refl"] & \Tm
    \end{tikzcd}
    \]
      The term $\refl_a$ has the correct type by the commutativity of 
    \[
    \begin{tikzcd}
        & \Tm\ar[r,"\refl"]\ar[d,"\delta"] & \Tm\ar[d,"\tp"] \\
        \Gamma \ar[ur,"{a}"] \ar[r,"{a\times_\Tm a}", swap]& \Tm\times \Tm\ar[r,"\Id", swap] & \Ty
    \end{tikzcd}
    \]
    \item
      Stability of $\refl_a$ under substitution is obvious.
    \item
      Suppose $C : \Gamma .(x : A) . \Id_A(a,x) \to \Ty$
      is a motive with $c_\refl : \Gamma \to \Tm$ satisfying
      \[ c_\refl \gg \tp = \rho_a \gg C\]
      We can first recognize the elementary
      gadgets as pullbacks of their algebraic counterparts.
    \[\begin{tikzcd}
	\Gamma & \Tm \\
	{\Gamma . (x : A) . \Id(a,x)} & {\Tm . \tp . \Id} \\
	{\Gamma.(x : A)} & {\Tm . \tp} \\
	\Gamma & \Tm
	\arrow["a", from=1-1, to=1-2]
	\arrow["{\rho_a}"', from=1-1, to=2-1]
	\arrow["\lrcorner"{anchor=center, pos=0.125}, draw=none, from=1-1, to=2-2]
	\arrow["\rho", from=1-2, to=2-2]
	\arrow[from=2-1, to=2-2]
	\arrow["{d_{\Id_A(a,x)}}"', from=2-1, to=3-1]
	\arrow["\lrcorner"{anchor=center, pos=0.125}, draw=none, from=2-1, to=3-2]
	\arrow["{d_\Id}", from=2-2, to=3-2]
	\arrow[from=3-1, to=3-2]
	\arrow["{d_A}"', from=3-1, to=4-1]
	\arrow["\lrcorner"{anchor=center, pos=0.125}, draw=none, from=3-1, to=4-2]
	\arrow["{d_\tp}", from=3-2, to=4-2]
	\arrow["a"', from=4-1, to=4-2]
    \end{tikzcd}\]
      By the universal property of $P_i X$ in \Cref{prop:poly-up},
      a pair of maps $a : \Gamma \to \Tm$
      and $\gamma : \Gamma . (x : A) . \Id_A (a,x) \to X$
      \[\begin{tikzcd}
	X & {\Gamma . (x : A) . \Id(a,x)} & {\Tm . \tp . \Id} \\
	& {\Gamma.(x : A)} & {\Tm . \tp} \\
	& \Gamma & \Tm
	\arrow["\gamma"', from=1-2, to=1-1]
	\arrow[from=1-2, to=1-3]
	\arrow["{d_{\Id_A(a,x)}}"', from=1-2, to=2-2]
	\arrow["\lrcorner"{anchor=center, pos=0.125}, draw=none, from=1-2, to=2-3]
	\arrow["{d_\Id}", from=1-3, to=2-3]
	\arrow["i", bend left = 60, from=1-3, to=3-3]
	\arrow[from=2-2, to=2-3]
	\arrow["{d_A}"', from=2-2, to=3-2]
	\arrow["\lrcorner"{anchor=center, pos=0.125}, draw=none, from=2-2, to=3-3]
	\arrow["{d_\tp}", from=2-3, to=3-3]
	\arrow["a"', from=3-2, to=3-3]
\end{tikzcd}\]
      is equivalent to a morphism
      $(a, \gamma) : \Gamma \to P_i X$.
      Then, using the computation of the
      natural transformation $v$ in
      \Cref{prop:vertical-nat-trans-computation} we
      make a cone for the weak pullback diagram
    \[\begin{tikzcd}
	\Gamma & {\Tm \times \Tm} \\
	{P_i\Ty} & {\Tm \times \Ty}
	\arrow["{(a, c_\refl)}", from=1-1, to=1-2]
	\arrow["{(a,C)}"', from=1-1, to=2-1]
	\arrow["{(a, \rho_a\gg C)}"{description}, from=1-1, to=2-2]
	\arrow["{{\Tm \times \tp}}", from=1-2, to=2-2]
	\arrow["{{v_\Ty}}"', from=2-1, to=2-2]
    \end{tikzcd}\]
      Producing a (coherent) lift (see 
      \Cref{def:algebraic-id-types} part (3))
      \[(a,j) : \Gamma \to P_i \Tm\]
      by the weak pullback property.
      By \Cref{prop:vertical-nat-trans-computation},
      we obtain
      \[j : \Gamma . (x : A) . \Id_A (a,x) \to \Tm\]
      satisfying $j \gg \tp = C$ and
      $\id_\Gamma . a . \refl \, a \gg j (C, c_\refl) = c_\refl$.
  \item 
    If the weak pullback structure is coherent then
    stability of $j$ under substitution follows.
    Again,
    we can always correct the weak pullback structure to
    a coherent one by \Cref{lem:coherent-weak-pullback}.
\end{enumerate}
\end{proof}

\subsection{Extracting algebraic type formers from a \foreignlanguage{greek}{π}-clan}

Yet another kind of translation is provided in \cite{awodey2018},
where a natural model $\pi : \DD_1 \to \DD_0$
is constructed from a $\pi$-clan
(or a \emph{closed, stable class of maps} in \cite{awodey2018}).
This construction, referred to as \emph{strictification},
constructs a left adjoint of the inclusion functor from the category of pseudofunctors to that of strict functors. 
It converts the pseudofunctor characterized by a $\pi$-clan $(\CC,\RR)$
\[\RR(-) : \CC^\op \to \Cat\]
into an equivalent (strict) {functor} $\CC^\op \to \Cat$.
In \cite[Proposition 28]{awodey2018},
the constructed natural model is shown to have
$\Sigma$- and $\Pi$-types.
With the further assumption of the $\pi$-clan being a ``$\pi$-tribe'' \cite{joyal2017}
(or ``factorizing'' in \cite{awodey2018}),
the constructed natural model is shown to also have $\Id$-types \cite[Proposition 30]{awodey2018}.
To ensure stability of type formers under substitution,
both proofs rely on the full structure of $\pi$-clan,
including pushforwards.

A similar construction is used in \cite[Theorem 2.3.4.]{kapulkin2021},
where a strict model of MLTT in simplicial sets is constructed by
combining a universe with certain closure conditions
on the maps classified by the universe.

To find an analog of \cite[Proposition 28]{awodey2018} and
\cite[Theorem 2.3.4.]{kapulkin2021} in our setting,
we look for sufficient preclan conditions for 
having the structure of algebraic type formers on a universe.
This can also be thought of as the converse of
\Cref{thm:elementary-universe-to-algebraic-universe}.
Unlike \cite[Proposition 28]{awodey2018},
\Cref{prop:principal-pi-clan-to-algebraic}
will not carry with it a ``strictification'' result:
we assume a universe $\tp$ as a hypothesis,
which already provides a strict interpretation of types.

\begin{proposition}
  \label{prop:principal-pi-clan-to-algebraic}
  Suppose $(\CC,\RR)$ is a $\pi$-clan and $\tp : \Tm \to \Ty$
  is a $(\CC,\RR)$-algebraic universe
  (\Cref{def:algebraic-universe}).
  Let us use $\RR_\tp$ to denote the principal class
  generated by $\tp$ (\Cref{def:principal-class}).
  \begin{itemize}
    \item If $\CC$ has a terminal object and $\RR_\tp$ contains all
      isomorphisms, then there exists\footnote{
        Unfortunately,
        the condition of having a pullback square in the 
        definition of $\RR_\tp$ is merely propositional;
        one could rephrase the theorem more constructively.
      } an algebraic $\Unit$-type structure on $\tp$
      (as a $(\CC,\RR)$-algebraic universe).
    \item If $\RR_\tp$ is closed under composition,
      then there exists an algebraic $\Sigma$-type structure on $\tp$.
    \item If $\RR_\tp$ is closed under pushforward,
      then there exists an algebraic $\Pi$-type structure on $\tp$.
  \end{itemize}
  Combined, if $\RR_\tp$ is a $\pi$-preclan, then $\tp$ models $\Unit$-,
  $\Sigma$-, and $\Pi$-types.
\end{proposition}
\begin{proof}
  ($\Unit$-types.) The case of $\Unit$-types follows from the definition of $\RR_\tp$,
  and the assumption applied to the identity map on the terminal object.

  ($\Sigma$-types.) Suppose that $\RR_\tp$ is closed under composition.
  Consider the diagram defining the polynomial composition
  $\tp \pcomp \tp$ (\Cref{def:pcomp}).
    \[
    \begin{tikzcd}
    \Tm & \compDom \\
    \Ty & \bullet & \Tm \\
    & {P_\tp \Ty} & \Ty
    \arrow["\tp"', from=1-1, to=2-1]
    \arrow[from=1-2, to=1-1]
    \arrow["\lrcorner"{anchor=center, pos=0.125, rotate=-90}, draw=none, from=1-2, to=2-1]
    \arrow["{t_2}"', from=1-2, to=2-2]
    \arrow["{{\tp \pcomp \tp}}"{description, pos=0.2}, shift left=5, bend left = 50, from=1-2, to=3-2]
    \arrow["{{\sndProj}}", from=2-2, to=2-1]
    \arrow[from=2-2, to=2-3]
    \arrow["{t_1}"', from=2-2, to=3-2]
    \arrow["\lrcorner"{anchor=center, pos=0.125}, draw=none, from=2-2, to=3-3]
    \arrow["\tp", from=2-3, to=3-3]
    \arrow["{{\fstProj}}"', from=3-2, to=3-3]
  \end{tikzcd}\]
  By the definition of the principal class,
  the pullbacks $t_1$ and $t_2$ are in $\RR_\tp$.
  Hence their composition $\tp \pcomp \tp$ is also in $\RR_\tp$.
  Thus there exist two maps $\pair: \compDom \to \Tm$ and
  $\Sigma : P_\tp \Ty \to \Ty$,
  forming a pullback square
  \[
  \begin{tikzcd}
      \compDom \ar[d,"\tp \pcomp \tp",swap]\ar[r,"\pair"] & \Tm \ar[d,"\tp"]\\
      P_\tp\Ty \ar[r,swap,"\Sigma"]& \Ty
  \end{tikzcd}
  \]

  ($\Pi$-types.) Suppose $\RR_\tp$ is closed under pushforward.
  Like with the previous case,
  it suffices to show that the polynomial application
  $P_\tp \tp : P_\tp \Tm \to P_\tp \Ty$ is in $\RR_\tp$.
  We unfold the definition of the polynomial functor
  \[ P_\tp X = \Ty_! (\tp_* (\Tm^* X))\]
  Firstly, $h := \Tm^* \tp : \Tm^* \Tm \to \Tm^* \Ty$ in $\RR(\Tm)$
  has an $\RR_\tp$-map as its underlying map in $\CC$,
  since $\RR_\tp$ is stable under pullback.
  Since
  \[P_\tp \tp = \Ty_! (\tp_* h) \]
  is the underlying map of $\tp_* h$,
  it suffices to show in general that $\tp_* : \RR(\Tm) \to \RR(\Ty)$
  preserves $\RR_\tp$-maps between objects in $\RR(\Tm)$.
  Let $g : Z \to \Tm$ be an $\RR$-map,
  and let $h : W \to Z$ be an $\RR_\tp$-map.
  Write $q := \tp_* g$.
  We obtain a diagram like \Cref{lem:pushforward-along-pullback-morphism-action}.
  \[
  \begin{tikzcd}
    Z & {\tp \times_X q} & {\tp_* Z} \\
    & \Tm & \Ty
    \arrow["g"', from=1-1, to=2-2]
    \arrow["{{\epsilon}}"', from=1-2, to=1-1]
    \arrow["{{f'}}", from=1-2, to=1-3]
    \arrow[from=1-2, to=2-2]
    \arrow["\lrcorner"{anchor=center, pos=0.125}, draw=none, from=1-2, to=2-3]
    \arrow["q", from=1-3, to=2-3]
    \arrow["\tp"', from=2-2, to=2-3]
  \end{tikzcd}
  \]
  We apply
  \Cref{cor:pushforward-stability}.
  Firstly, $h : W \to Z$ is an $\RR$-map
  between objects $W$ and $Z$ in $\RR(Y)$,
  since $h \in \RR_\tp \subseteq \RR$ and
  $\RR$ is closed under composition.
  Furthermore, $\RR_\tp$ is stable under pullback,
  so it is stable under
  pre- and post-composition with isomorphisms.
  Hence,
  we have $\tp_* h$ is an $\RR_\tp$-map if and only if  
  the object $f'_* (\epsilon^* h)$ in $\RR(\tp_* Z)$
  has an underlying map in $\RR_\tp$.
  Indeed, $\RR_\tp$ is stable under pullback so
  $\epsilon^* h$ is in $\RR_\tp(\tp \times_X q)$.
  The pullback $f'$ is in $\RR_\tp$
  and $\RR_\tp$ is closed under pushforward,
  thus $f'_* (\epsilon^* h)$ in $\RR(\tp_* Z)$.

\end{proof}

\begin{example}
    
Compared with the current formalized construction, \Cref{prop:principal-pi-clan-to-algebraic}
is closer in presentation to models of 
homotopy type theory (HoTT)
appearing in the literature.
For example, this is the approach to defining the simplicial model
of HoTT in \cite[Theorem 2.3.4]{kapulkin2021},
and the cubical model of HoTT in \cite{awodey2023}.
In both cases a universe is constructed,
the universe classifies a certain class of maps 
(fibrations with small fibers),
and that class of maps is proven to be closed under certain operations.

\end{example}

We take great care in \Cref{prop:principal-pi-clan-to-algebraic}
to not assume too many conditions on $\RR_\tp$.
If $(\CC,\RR_\tp)$ were a clan,
then it would suffer from \textsf{Type : Type} issues.
To address this,
while also being able to express the polynomial $P_\tp \Ty$,
we assumed the presence of a larger $\pi$-clan $(\CC,\RR)$,
with $\RR_\tp \subseteq \RR$.

\begin{example}
  We apply \Cref{prop:principal-pi-clan-to-algebraic}
  to the groupoid model of MLTT \cite{hofmann1995}.
  Here $\CC = \Grpd$ is the category of large groupoids
  (fixing a notion of size) and
  $\RR$ is the class of split isofibrations in $\Grpd$.
  We take $\tp : \Tm \to \Ty$ to be the classifier of small 
  split isofibrations\footnote{
    Explicitly, 
    the classifier can be built by applying the
    Grothendieck construction to the
    inclusions $\mathsf{Core(grpd)} \to \Grpd \to \Cat$.
    Here $\mathsf{Core(grpd)}$ is the core of the
    category of small groupoids.
    A similar construction in $\Cat$ can be found in \cite{hess2007}.
  },
  i.e. split isofibrations with small fibers.
  With a careful treatment of size,
  one can show that the class of small
  split isofibrations $\RR_\tp$
  contains all isomorphisms, is closed under composition,
  and is closed under pushforward.

  In the HoTTLean project \cite{hua2025}, 
  the groupoid model of MLTT \cite{hofmann1995}
  is constructed by defining universes classifying small 
  split isofibrations,
  as described above.
  However, strict type formers on the universe
  are constructed in terms of elementary semantics.
  Although universes are straightforward to define, 
  defining $\Sigma$- and $\Pi$-type structures on
  those universes by hand leads to working with many
  equalities between types,
  which is best avoided if possible
  \cite{zulipthread2023}.
  \Cref{prop:principal-pi-clan-to-algebraic}
  suggests a different approach to defining the groupoid model
  that may help address these problems.
  In future work, we plan to experiment with this different approach.
\end{example}

\section{Conclusion}\label{sec:Concl}
We have presented an approach to defining semantics of MLTT
in $\pi$-clans, called algebraic semantics,
generalizing the LCC analog in the literature.
The key step of this approach is to make polynomial functors available in this setting,
as they are central to defining algebraic type formers.
Once the general theory of polynomial functors in $\pi$-clans is established,
the rest of our approach is analogous to its LCC counterpart,
as described in \cite{awodey2025}.

Furthermore, we validated the correctness of our definition by proving that it models the type formers in the expected way. 
This was achieved by proving that algebraic semantics are 
inter-translatable with an elementary presentation of a strict MLTT
semantics called elementary semantics.
The efficacy of our approach was witnessed by its application in HoTTLean.
In this library,
we defined elementary and algebraic semantics,
formalizing translations between the two, so as to secure the correctness of the elementary definition that we used to define the groupoid model.

\section{Related Work}
\subsection{Polynomial functors}
Weber \cite{weber2015} developed polynomial functors in
the general setting of a category with pullbacks and exponentiable maps.
Aberl\'e and Spivak \cite{aberle2025} study the theory of 
polynomial functors and their application for modeling type 
formers of MLTT in a univalent setting,
with results formalized in Agda.
Our polynomial functor library is based on the Lean library Poly
\cite{2024hazratpour_poly},
which formalizes polynomial functors in Weber's setting.
Uemura \cite{uemura2022} characterized
exponentiable maps as models of a universal one,
and characterizes polynomial endofunctors
in terms of endofunctors equipped with certain natural transformations.

\subsection{Clans}
Much of the theory that underlies our development of polynomial
functors in $\pi$-clans is given by
Taylor \cite{taylor1987} and Joyal \cite{joyal2017}.
The key difference between these clan-based descriptions of MLTT 
semantics and ours is our use of a universe to ensure strict coherence.
More recently,
Frey \cite{frey2025} proved a Gabriel--Ulmer style duality theorem for clans,
rounding out the picture of clans as an internal language with
functorial semantics and theory-model duality.

\subsection{Semantics of MLTT}
MLTT semantics are presented in various styles,
including
display map categories \cite{taylor1987} (or clans),
comprehension categories \cite{jacobs1993},
categories with families \cite{dybjer1997},
and natural models \cite{awodey2018}.
Some of these presentations are compared in \cite{ahrens2018}, in a univalent setting,
but without considering type formers.
A formalization of these equivalences can be found in
the UniMath library \cite{2024grayson}.
Our work also draws from 
Uemura's categories with representable maps \cite{uemura2023},
which provides an analysis of natural models
via Lawvere's \emph{functorial semantics} \cite{lawvere1963}.

\section{Future work}\label{sec:Future}
We have set out the basic theory of polynomial functors in $\pi$-clans
with only the goal of formulating the algebraic semantics of type theory.
The full picture of this theory is of general interest
and deserves further exploration;
a detailed treatment will be provided in the PhD thesis of the
first author.


Whereas our work is already a generalization from LCC categories,
it has the potential of being generalized further. 
Using an idea similar to \cite{weber2015},
where polynomial signatures are taken to be exponential morphisms
in finite limit categories, 
instead of taking
a polynomial signature to be any morphism in a $\pi$-clan, 
one can possibly take it to be a morphism that is
``exponentiable relative to a preclan''.

Uemura \cite{uemura2022} characterized polynomial functors
and exponentiability using the language of
\emph{generalized algebraic theories} in finite limit categories. 
This could potentially be reformulated in terms of preclans.
Furthermore,
we would like to reformulate Uemura's 
analysis of natural models \cite{uemura2023}
in the style of Lawvere's functorial semantics
\cite{lawvere1963},
producing a counterpart in terms of our algebraic models.


\bibliography{main.bib}
\bibliographystyle{alpha}
\end{document}